\documentclass[a4paper]{amsart}
\usepackage[T1]{fontenc}
\usepackage[latin1]{inputenc}
\usepackage[english]{babel}
\usepackage{amsmath,amsthm,latexsym,amssymb,mathrsfs}
\DeclareMathAlphabet{\altmathcal}{OMS}{cmsy}{m}{n}
\usepackage{mathptmx}
\usepackage{dsfont} 
\usepackage{mathtools}
\usepackage{color}
\usepackage{physics}
\usepackage{faktor}
\usepackage[shortlabels]{enumitem}
\usepackage{hyperref}
\usepackage{graphicx}
\usepackage{tabularx}
\newcolumntype{C}{>{\centering\arraybackslash}X} 
\usepackage{csquotes}
\usepackage{url}
\usepackage{import}
\usepackage[pass]{geometry}

\usepackage[style=alphabetic,natbib=true,url=false,doi=false,isbn=false]{biblatex}
\renewbibmacro{in:}{}
\newtheorem{theorem}{Theorem}[section]
\newtheorem*{theorem*}{Theorem}
\newtheorem{theoremx}{Theorem}
\newtheorem{lemma}[theorem]{Lemma}
\newtheorem{corollary}[theorem]{Corollary}
\newtheorem{proposition}[theorem]{Proposition}
\theoremstyle{definition}
\newtheorem{definition}[theorem]{Definition}

\theoremstyle{remark}
\newtheorem{remark}[theorem]{Remark}

\newcommand{\defin}{\vcentcolon =}
\newcommand{\C}{\mathbb{C}}
\newcommand{\R}{\mathbb{R}}
\newcommand{\N}{\mathbb{N}}
\newcommand{\Z}{\mathbb{Z}}
\newcommand{\I}{I}
\newcommand{\II}{I  \! \! I}
\newcommand{\III}{I \! \! I \! \! I}
\newcommand{\Hyp}{\mathbb{H}}
\newcommand{\dS}{\mathrm{dS}}
\newcommand{\Teich}{\altmathcal{T}}

\newcommand{\length}{\ell}
\newcommand{\CC}{C}

\newcommand{\WP}{\textit{WP}}

\newcommand{\id}{\textit{id}}
\newcommand{\hyp}{\mathfrak{h}}
\newcommand{\conf}{\mathfrak{c}}
\newcommand{\1}{\mathds{1}}

\newcommand{\MesLam}{\altmathcal{ML}}
\newcommand{\QF}{\altmathcal{QF}}
\newcommand{\mappa}[3]{#1 \colon #2 \rightarrow #3}
\newcommand{\hsk}{\hskip0pt}
\newcommand{\RVol}{V_R}

\DeclarePairedDelimiterX{\scal}[2]{\langle}{\rangle}{#1, #2}
\DeclarePairedDelimiterX{\scall}[2]{(}{)}{#1, #2}
\DeclarePairedDelimiter{\set}{\{}{\}}
\DeclareMathOperator{\arcsinh}{arcsinh}

\DeclareMathOperator*{\supess}{ess~sup}
\DeclareMathOperator*{\injrad}{injrad}

\DeclareMathOperator{\Vol}{Vol}

\DeclareMathOperator{\Proj}{P}

\DeclareMathOperator{\SL}{SL}
\DeclareMathOperator{\ext}{ext}
\DeclareMathOperator{\grd}{grad}

\DeclareMathOperator{\sgr}{sgr}

\DeclareMathOperator{\dvol}{dvol}

\title[The dual volume of quasi-Fuchsian manifolds and the WP-distance]{The dual volume of quasi-Fuchsian manifolds and the Weil-Petersson distance}

\author{Filippo Mazzoli}

\thanks{Supported by the Luxembourg National Research Fund PRIDE15/10949314/GSM/Wiese.}

\date{\today}

\address{Department of Mathematics, 
	University of Virginia
	Charlottesville, VA, 
	United States}

\email{filippomazzoli@me.com}

\begin{document}

\begin{abstract}
	\noindent Making use of the dual Bonahon-\hsk Schl\"afli formula, we prove that the dual volume of the convex core of a quasi-\hsk Fuchsian manifold $M$ is bounded by an explicit constant, depending only on the topology of $M$, times the Weil-\hsk Petersson distance between the hyperbolic structures on the upper and lower boundary components of the convex core of $M$. 
\end{abstract}
	
\maketitle

\section*{Introduction}

Let $\Sigma$ be a closed oriented surface of genus $g \geq 2$. The Teichm\"uller space of $\Sigma$, denoted by $\Teich(\Sigma)$, can be interpreted as the space of isotopy classes of either conformal structures or hyperbolic metrics on $\Sigma$, thanks to the uniformization theorem. The Weil-\hsk Petersson K\"ahler structure, with its induced distance $d_\WP$, naturally arises from the interplay of these two interpretations. As described by \citet{brock2003the_weil}, the coarse geometry of the Weil-\hsk Petersson distance turns out to be related to the growth of the volume of the convex core of quasi-\hsk Fuchsian manifolds. More precisely, in \cite{brock2003the_weil} the author proved the existence of two constants $K_1 > 1$ and $K_2 > 0$, depending only on the topology of $\Sigma$, such that every quasi-\hsk Fuchsian manifold $M$ satisfies
\begin{equation} \label{eq:brock_bounds}
K_1^{-1} d_\WP(c^+(M),c^-(M)) - K_2 \leq \Vol(\CC M) \leq K_1 \ d_\WP(c^+(M),c^-(M)) + K_2 .
\end{equation}
Inspired by this phenomenon, the aim of this paper is to determine an explicit control from above of the \emph{dual volume} of the convex core of a quasi-Fuchsian manifold $M$ in terms of the Weil-\hsk Petersson distance between the \emph{hyperbolic metrics} on the boundary of its convex core, in analogy to what has been done by \citet{schlenker2013renormalized} for the \emph{renormalized volume} of $M$ and the Weil-Petersson distance between its \emph{conformal structures} at infinity.

In order to be more precise, we need to introduce some notation. If $M$ is a quasi-\hsk Fuchsian manifold homeomorphic to $\Sigma \times \R$, then $\CC M$ will denote its convex core. When $M$ is not Fuchsian, the subset $\CC M$ is homeomorphic to the product of $\Sigma$ with a compact interval of $\R$ with non-\hsk empty interior. Its boundary components $\partial^\pm \CC M$, are locally convex pleated surfaces with hyperbolic metrics $m^+(M)$, $m^-(M) \in \Teich(\Sigma)$ and bending measured laminations $\mu^+$, $\mu^-$. The manifold $M$ can be extended at infinity by adding two surfaces $\partial_\infty^\pm M$, so that $M \cup \partial_\infty^\pm M$ is homeomorphic to $\Sigma \times [- \infty, + \infty]$. The surfaces $\partial_\infty^\pm M$ are endowed with natural complex structures $c^\pm(M)$, coming from the conformal action of the fundamental group of $M$ on the boundary at infinity of the hyperbolic space $\Hyp^3$. By the works of Ahlfors and Bers (see \cite{bers1960simultaneous}, \cite{ahlfors1960riemanns}), the data of $c^\pm(M)$ uniquely determine the hyperbolic manifold $M$, and any couple of conformal structures can be realized in this way.

The notion of \emph{dual volume} arises from the polar correspondence between the \emph{hyperbolic $3$-\hsk space} $\Hyp^3$ and the \emph{de Sitter $3$-\hsk space} $\dS^3$ (see \cite{rivin1986phd}, \cite{schlenker2002hypersurfaces} for details), In general, if $N$ is a convex subset with regular boundary of a quasi-\hsk Fuchsian (or, in general, co-\hsk compact hyperbolic) manifold $M$, we set the \emph{dual volume} of $N$ to be
\[
\Vol^*(N) \defin \Vol(N) - \frac{1}{2} \int_{\partial N} H \dd{a} ,
\]
where $H$ is the trace of the shape operator of $\partial N$ with respect the interior normal vector field. Even if the convex core does not have regular boundary, a simple approximation argument shows that it makes sense to define the dual volume of the convex core of $M$ as $V_\CC^*(M) \defin \Vol(\CC M) - \frac{1}{2} L_\mu(m)$, where $\mu$ is the bending measured lamination of $\partial \CC M$ and $m \in \Teich(\partial \CC M)$ is its hyperbolic metric. As we deform the quasi-\hsk Fuchsian structure $(M_t)_t$, the variation of the dual volume of the convex core is described by the \emph{dual Bonahon-\hsk Schl{\"a}fli formula}, which asserts that
\[
\dd{V_\CC^*}(\dot{M}) = - \frac{1}{2} \dd{L_\mu}(\dot{m}) .
\]
Here $\dot{m}$ denotes the derivative of the hyperbolic metrics $m_t$ on the boundary of the convex cores $\CC M_t$, and $L_\mu$ stands for the hyperbolic length function of $\mu$ over the Teichm\"uller space of $\partial \CC M$. An elementary proof of this relation, originally showed in \cite{krasnov2009symplectic}, can be found in \cite{mazzoli2021the_dual}.

The dual Bonahon-\hsk Schl\"afli formula, together with the properties of the bending measured lamination, allows us to bound uniformly the variation of $V_\CC^*(M)$ with respect to $\dot{m}$. The aim of this paper is to prove the following statement:
\begin{theoremx}\label{thm:bound_dual_volume_WP}
	There exists a universal constant $C > 0$ such that, for every quasi-\hsk Fuchsian manifold $M$ homeomorphic to $\Sigma \times \R$,  we have
	\[
	\abs{V_\CC^*(M)} \leq C \ (g - 1)^{1/2} \ d_\WP(m^-(M),m^+(M)) ,
	\]
	with $C \approx 7.3459$.
\end{theoremx}

The dual volume and the hyperbolic volume of the convex core differ by the term $ \frac{1}{2} L_\mu(m)$, which is bounded by $6 \pi \abs{\chi(\Sigma)}$ (see \cite{bridgeman_brock_bromberg2017}). By the work of \citet{sullivan1981travaux} (see also Epstein-Marden \cite{epstein2006fundamentals} and Bridgeman-Canary-Yarmola \cite{yarmola2016improved}), the Teichm\"uller distance between the conformal structures at infinity $c^\pm(M)$ and the hyperbolic structures $m^\pm(M)$ is bounded from above by an explicit universal constant. Since the Weil-Petersson distance is controlled from above by $\sqrt{2 \pi \abs{\chi(\Sigma)}}$ times the Teichm\"uller distance (see \citet{linch1974a_comparison}), the combination of Theorem \ref{thm:bound_dual_volume_WP} with the above mentioned results provides an alternative proof of Brock's upper bound in \eqref{eq:brock_bounds} and it produces explicit additive and multiplicative constants, with a fairly simple argument.

A similar strategy has been developed by \citet{schlenker2013renormalized} to study the behavior of the \emph{renormalized volume} of quasi-\hsk Fuchsian manifolds. The notion of renormalized volume has been initially introduced by \citet{witten1998antidesitter} and independently by \citet{gubser1998gaugetheory} in the context of conformally compact Einstein manifolds, motivated by the AdS/CFT correspondence of string theory proposed by \citet{maldacena1998thelarge}. Later \citet{krasnov2008renormalized} defined such quantity for convex co-\hsk compact hyperbolic $3$-\hsk manifolds, emphasizing its relations with the geometry of the Teichm\"uller space (compare also with the work of \citet{zongraf1987ontheuniformization}). The key ingredients in the work of \citet{schlenker2013renormalized} are the variation formula of the renormalized volume $\RVol(M)$ and the Nehari's bound of the norm of the Schwarzian derivative of the complex projective structures at infinity of $\partial_\infty M$. In particular, the author showed that, for every quasi-Fuchsian manifold $M$
\begin{equation} \label{eq:bound_renorm_WP}
\RVol(M) \leq 3 \sqrt{\pi} (g - 1)^{1/2} \ d_\WP(c^+(M),c^-(M)) .
\end{equation}
We remark that the multiplicative constant $C$ appearing in our statement is larger than the one obtained using the renormalized volume, $3 \sqrt{\pi} \approx 5.3174 < 7.3459 \approx C$. Therefore, the inequality \eqref{eq:bound_renorm_WP} is more efficient in terms of coarse estimates.

Nevertheless, Theorem \ref{thm:bound_dual_volume_WP} carries more information than its implications concerning the coarse Weil-\hsk Petersson geometry, in particular when we consider quasi-\hsk Fuchsian structures that are close to the Fuchsian locus. In this case, Theorem \ref{thm:bound_dual_volume_WP} and the inequality \eqref{eq:bound_renorm_WP} furnish complementary insights, since they involve the Weil-\hsk Petersson distance between the hyperbolic structures on the boundary of the convex core, on one side, and the conformal structures at infinity on the other. Moreover, Proposition \ref{prop:bound_differential_length} and its application for the bound of the dual volume show that the multiplicative constant in Theorem \ref{thm:bound_dual_volume_WP} can be improved if we have a better control of $L_\mu(m)$ than $L_\mu(m) \leq 6 \pi \abs{\chi(\Sigma)}$ (from \cite{bridgeman_brock_bromberg2017}), exactly as the inequality \eqref{eq:bound_renorm_WP} can be improved if we have a better control of the $L^\infty$-\hsk norm of the Schwarzian at infinity than the Nehari's bound. 

We finally mention that, carrying on the analogy between the picture "at the convex core" and "at infinity" by \citet{schlenker2017notes}, our result fits well into the comparison of the two descriptions of the space of quasi-\hsk Fuchsian structures, as summarized in the following table:

\begin{table}[ht]
	\centering
	\def\arraystretch{1.1}
	\begin{tabular}{|c|c|}
		\hline
		On $\partial \CC M$ & On $\partial_\infty M$ \\
		\hline\hline
		Induced metrics $m^\pm$ & Conformal structures $c^\pm$ \\
		\hline
		Thurston's conjecture & Bers' Simultaneous \\
		on prescribing $m^\pm$ & Uniformization Theorem \\
		\hline 
		Bending measured lamination $\mu$ & Measured foliation $\altmathcal{F}$ \\
		\hline
		Hyperbolic length $L_\mu(m)$ & Extremal length $\ext_\altmathcal{F}(c)$ \\
		\hline
		Dual volume $V_\CC^*(M)$ & Renormalized volume $\RVol(M)$ \\
		\hline
		Dual Bonahon-Schl\"afli formula & \cite[Theorem~1.2]{schlenker2017notes}  \\
		$\dd V_\CC^* = -\frac{1}{2} \dd{L_\mu}(\dot m)$ & $\dd \RVol = -\frac{1}{2} \dd{\ext_\altmathcal{F}}(\dot c)$ \\
		\hline
		Bound on $L_\mu(m)$ \cite{bridgeman_brock_bromberg2017} & \cite[Theorem~1.4]{schlenker2017notes} \\
		$L_{\mu}(m)\leq 6\pi|\chi(S)|$ & $\ext_{\altmathcal{F}}(c) \leq 3\pi |\chi(S)|$ \\
		\hline
		Bound of $V_\CC^*$ with $d_\WP(m^+,m^-)$ & Bound of $\RVol$ with $d_\WP(c^+,c^-)$ \\
		Theorem \ref{thm:bound_dual_volume_WP} & Inequality \eqref{eq:bound_renorm_WP} \\
		\hline
	\end{tabular}
\end{table}

\subsection*{Outline of the paper} 
In Section \ref{section:preliminaries} we recall the definition of Teichm\"uller space $\Teich(\Sigma)$ as deformation space of Riemann surface structures, and of its tangent and cotangent bundles via Beltrami differentials and holomorphic quadratic differentials. Then, following \cite{tromba2012teichmuller}, we introduce the description of $\Teich(\Sigma)$ as the space of isotopy classes of hyperbolic metrics, and of its tangent bundle using traceless and divergence free (also called transverse traceless) symmetric tensors. The Section ends with a simple Lemma describing the relation between the two equivalent interpretations and between their norms.
	
Section \ref{section:bound_differential_length} is devoted to the proof of Proposition \ref{prop:bound_differential_length}, in which we produce a uniform bound of the differential of $\mappa{L_\mu}{\Teich(\Sigma)}{\R}$, the hyperbolic length function of a measured lamination over the Teichm\"uller space. This is the main "quantitative" ingredient for the proof of Theorem \ref{thm:bound_dual_volume_WP}. The proof uses Tromba's description of $T \Teich(\Sigma)$ via transverse traceless tensors: we represent a variation of hyperbolic metrics $\dot{m}$ as the real part of a holomorphic quadratic differential $\Phi$. Using standard properties of holomorphic functions, the pointwise norm of $\Phi$ at $x$ can be bounded by the $L^\emph{p}$-\hsk norm of $\Phi$ over some embedded geodesic ball in $\Sigma$ centered at $x$. The variation of $L_\mu$ can be expressed as an integral over the support of $\mu$ of the product of the variation of the length measure of $\dot{m}$ times the tranverse measure of $\mu$. Then the result will follow by using the pointwise estimation and a Fubini's exchange of integration over a suitable finite cover of $\Sigma$.

In Section \ref{section:dual_volume} we obtain a uniform control of the differential of $V_\CC^*$ in terms of the norm of the variation of the hyperbolic metrics on $\partial \CC M$. To do so, we will apply the works of \citet{yarmola2016improved} and \citet{bridgeman_brock_bromberg2017}, which give universal controls of the bending measure of the convex core. These results are to the dual volume as the Nehari's bound of the norm of the Schwarzian derivative is to the renormalized volume (the bounds obtained in \cite{bridgeman_brock_bromberg2017} are actually proved \emph{using} Nehari's bound). The dual Bonahon-\hsk Schl\"afli formula relates the variation of $V_\CC^*$ with the differential of the length of the bending measured lamination, and the mentioned universal bounds combined with Proposition \ref{prop:bound_differential_length} will produce the desired control of $\dd{V_\CC^*}$ (see Corollary \ref{cor:bound_diff_dual_vol}).

In Section \ref{section:dual_volume_WP_distance} we will finally give a proof of Theorem \ref{thm:bound_dual_volume_WP}. Contrary to what happens for the conformal structures at infinity, the hyperbolic structures on $\partial \CC M$ are only conjecturally thought to give a parametrization of the space of quasi-Fuchsian manifolds. Because of this, proving Theorem \ref{thm:bound_dual_volume_WP} from Corollary \ref{cor:bound_diff_dual_vol} is not as immediate as it is for the renormalized volume using its variation formula. Our procedure to overcome this difficulty passes through the foliation of hyperbolic ends by constant Gaussian curvature surfaces $\Sigma_k$, with $k \in (-1,0)$, and the notion of landslide, which is a "smoother" analogue of earthquakes between hyperbolic metrics on $\Sigma$ introduced by \citet{bonsante2013a_cyclic} (see also \cite{bonsante2015a_cyclic}). By the work of \citet{schlenker2006hyperbolic} and \citet{labourie1991probleme}, the data of the metrics on the surfaces $\Sigma_k$ parametrize the space of quasi-Fuchsian manifolds. Therefore, the strategy will roughly be to:
\begin{enumerate}
	\item approximate the dual volume of the convex core $\CC M$ by the dual volume $V_k^*$ of the region enclosed by the $k$-surfaces of $M$;
	\item prove that the differentials of the functions $V_k^*$ converge to the differential of $V_\CC^*$ as $k$ goes to $-1$, i. e. as the surfaces $\Sigma_k$ get closer to the convex core $\CC M$;
	\item use the parametrization result for the metrics of $\Sigma_k$ to deduce the statement of Theorem \ref{thm:bound_dual_volume_WP} via an approximation argument.
\end{enumerate}
For point $(2)$, which is the most delicate part of our argument, we will highlight a connection between the differential of the functions $V_k^*$ and the infinitesimal smooth grafting, introduced in \cite{bonsante2013a_cyclic}. As described by \citet{mcmullen1998complex}, the earthquake map can be complexified using the notion of grafting along a measured lamination. In the same way the landslide admits a complex extension via the \emph{smooth grafting map}. Moreover, the complex earthquake can be actually recovered by a suitable limit of complex landslides. Using this convergence procedure, we are able to show that $\dd{V_\CC^*}$ is the limit of the differentials $\dd{V_k^*}$, in the sense described by Proposition \ref{prop:uniform_convergence_diff}. The rest of the proof of Theorem \ref{thm:bound_dual_volume_WP} will be an elementary application of the results from the previous section, similarly to what done in \cite{schlenker2013renormalized} with the renormalized volume.

\subsection*{Acknowledgments}
I am very grateful to Jean-Marc Schlenker for introducing me to this problem and for his help throughout this work. Moreover, I would like to thank the referee for the careful reading and for several useful comments that helped clarify the exposition. 

\section{Preliminaries} \label{section:preliminaries}

Let $\Sigma$ be an oriented closed surface of genus $g \geq 2$. Two Riemannian metrics $g$, $g'$ on $\Sigma$ are said to be \emph{conformally equivalent} if there exists a smooth function $\mappa{u}{\Sigma}{\R}$ such that $g = e^{2 u} g'$. A \emph{Riemann surface structure} on $\Sigma$ is a couple $X = (\Sigma,c)$, where $c$ is a conformal class of Riemannian metrics of $\Sigma$. A \emph{hyperbolic structure} on $\Sigma$ is the datum of a Riemannian metric $h$ with constant Gaussian curvature equal to $-1$.

The \emph{Teichm\"uller space} of $\Sigma$, denoted by $\Teich(\Sigma)$, is the space of isotopy classes of conformal structures over the surface $\Sigma$. Thanks to the uniformization theorem, the Teichm\"uller space can be considered equivalently as the space of isotopy classes of hyperbolic metrics on $\Sigma$. We will write $\Teich^\conf(\Sigma)$ ($\conf$ for \emph{conformal}) when we want to emphasize the first interpretation, and $\Teich^\hyp(\Sigma)$ ($\hyp$ for \emph{hyperbolic}) in latter case.

In the following, we will recall the definition of the Weil-Petersson Riemannian metric on the Teichm\"uller space. Since the literature usually agrees on its definition only up to multiplicative constant, we will spend some time in describing the setting we will work with, mainly because we will be interested in producing explicit bounds of our geometric quantities.

Let $X$ be a Riemann structure on $\Sigma$. A \emph{Beltrami differential} on $X$ is a $(1,1)$-\hsk tensor $\nu$ that can be expressed in local coordinates as $\nu = n \ \partial_z \otimes \dd{\bar{z}}$, where $n$ is a measurable complex-valued function. If $h = \rho \abs{\dd{z}}^2$ is the unique hyperbolic metric in the conformal class $c$, then for any $q \in [1,\infty)$ we define the $L^q$-\hsk norm of the Beltrami differential $\nu = n \ \partial_z \otimes \dd{\bar{z}}$ to be
\[
\norm{\nu}_{B, q} \defin \left( \int_X \abs{n}^q \rho \dd{x} \dd{y} \right)^{1/q} .
\]
When $q = \infty$, we set $\norm{\nu}_{B,\infty} \defin \supess_\Sigma \abs{n}$. We will denote by $B(X)$ the space of Beltrami differentials of $X$ with finite $L^\infty$-\hsk norm. Observe that the norm $\norm{\cdot}_{B, 2}$ on $B(X)$ is induced by the hermitian scalar product
\[
\scal{\nu}{\mu}_{B, 2} = \int_X \bar{n} m \ \rho \dd{x} \dd{y} ,
\]
where $\nu = n \ \partial_z \otimes \dd{\bar{z}}$ and $\mu = m \ \partial_z \otimes \dd{\bar{z}}$.

A holomorphic quadratic differential on $X$ is a symmetric $2$-\hsk covariant tensor that can be locally written as $\Phi = \phi \dd{z}^2$, where $\phi$ is holomorphic. In analogy to what was done above, for every $p \in [1,\infty)$ we define the $L^p$-\hsk norm of $\Phi$ to be
\[
\norm{\Phi}_{Q,p} \defin \left( \int_\Sigma \frac{\abs{\phi}^p}{\rho^{p -1}} \dd{x} \dd{y} \right)^{1/p} .
\]
When $p = \infty$, we set $\norm{\Phi}_{Q,\infty} \defin \supess_\Sigma \abs{\phi}/\rho$. When $p = 2$, the norm $\norm{\cdot}_{Q,2}$ is induced by a scalar product, defined as follows:
\[
\scal{\Phi}{\Psi}_{Q,2}  \defin \int_\Sigma \frac{\phi \overline{\psi}}{\rho} \dd{x} \dd{y} .
\]
There is a natural pairing between the space of bounded Beltrami differentials $B(X)$ and the space of holomorphic quadratic differentials $Q(X)$: given a Beltrami differential $\nu = n \ \partial_z \otimes \dd{\bar{z}}$ and a holomorphic quadratic differential $\Phi = \phi \dd{z}^2$, we define
\[
\scall{\Phi}{\nu} \defin \int_X \phi n \dd{x}\dd{y} .
\]
A Beltrami differential $\nu \in B(X)$ is \emph{harmonic} if there exists a holomorphic quadratic differential $\Phi = \phi \dd{z}^2$ such that $\nu = \bar{\phi} / \rho \ \partial_z \otimes \dd{\bar{z}}$. We denote by $B_h(X)$ the space of harmonic Beltrami differentials on $X$.

Let $N(X)$ be the subspace of $B(X)$ of those Beltrami differentials $\nu$ verifying $\scall{\Phi}{\nu} = 0$ for every $\Phi \in Q(X)$. As described in \cite{gardiner2000quasiconformal}, the space $B_h(\Sigma)$ and $N(X)$ are in direct sum, and the quotient of $B(X)$ by the subspace $N(X)$ identifies with the tangent space to the Teichm\"uller space $T_X \Teich^\conf(\Sigma)$ (here we denote by $X$ the isotopy class of the conformal structure, with abuse). Moreover, the pairing $\scall{\cdot}{\cdot}$ determines a natural isomorphism between the dual of $T_X \Teich^\conf(\Sigma)$ and the space of holomorphic quadratic differentials $Q(X)$, which is consequently identified with the cotangent space $T_X^* \Teich^\conf(\Sigma)$. The scalar product on $T_X \Teich^\conf(\Sigma)$ induced by $\Re \scal{\cdot}{\cdot}_{B,2}$ defines the \emph{Weil-\hsk Petersson metric} of the Teichm\"uller space $\Teich^\conf(\Sigma)$, and $\Re \scal{\cdot}{\cdot}_{Q,2}$ determines the corresponding metric on the cotangent bundle to Teichm\"uller space.

\begin{lemma} \label{lem:L2_dual_norm}
For every $\Phi \in Q(X)$ we have:
\[
\norm{\Phi}_{Q,2} = \sup_{\nu \in B(X) \setminus \set{0}} \frac{\abs{\scall{\Phi}{\nu}}}{\, \, \, \, \, \,\norm{\nu}_{B,2}} = \sup_{\nu \in B_h(X) \setminus \set{0}} \frac{\abs{\scall{\Phi}{\nu}}}{\, \, \, \, \, \, \norm{\nu}_{B,2}} .
\]
\begin{proof}
	By the Cauchy-\hsk Schwarz inequality we have $\abs{\scall{\Phi}{\nu}} \leq \norm{\Phi}_{Q,2} \norm{\nu}_{B,2}$, with equality realized by the harmonic Beltrami differential $\nu_\Phi$, which satisfies $(\Phi,\nu_\Phi) = \norm{\Phi}_{Q,2}^2 = \norm{\nu_\Phi}_{B,2}^2$. Therefore we get:
	\[
	\norm{\Phi}_{Q,2} \geq \sup_{\nu \in B(X) \setminus \set{0}} \frac{\abs{\scall{\Phi}{\nu}}}{\, \, \, \, \, \,\norm{\nu}_{B,2}} \geq \sup_{\nu \in B_h(X) \setminus \set{0}} \frac{\abs{\scall{\Phi}{\nu}}}{\, \, \, \, \, \, \norm{\nu}_{B,2}} \geq \norm{\Phi}_{Q,2} .
	\]
	The first inequality holds because of Cauchy-\hsk Schwarz, the second one because $B_h(X) \subset B(X)$, and the last one by taking $\nu = \nu_\Phi$.
\end{proof}
\end{lemma}

We recall now the Riemannian description of the Teichm\"uller space as developed in \cite{tromba2012teichmuller}. Let $S^{2,0} \Sigma$ be the bundle of $2$-\hsk covariant symmetric tensors on $\Sigma$, and let $\Gamma(S^{2,0} \Sigma)$ denote the space of its smooth sections, which is an infinite dimensional vector space. The space $\altmathcal{M}$ of smooth Riemannian metrics on $\Sigma$ identifies with an open cone inside $\Gamma(S^{2,0} \Sigma)$. Therefore, given any Riemannian metric $g$ on $\Sigma$, the tangent space $T_g \altmathcal{M}$ is canonically isomorphic to $\Gamma(S^{2,0} \Sigma)$. The metric $g$ determines a scalar product on $T_g \altmathcal{M}$, which can be expressed as $\scall{h}{k}_g \defin g^{i p} g^{j q} h_{i j} k_{p q}$, for $h$, $k$ in $\Gamma(T^{2,0} \Sigma)$. The norm induced by this scalar product will be denoted by $\norm{h}_g^2 \defin \scal{h}{h}_g$. Given $h \in \Gamma(T^{2,0} \Sigma)$, we define the $g$-\hsk divergence of $h$ to be the $1$-\hsk form $\delta_g h(V) \defin \trace_g(\nabla_* h)(*,V)$, for any $V$ tangent vector field to $\Sigma$. Now we set
\[
S_{tt}(\Sigma,g) \defin \set{h \in \Gamma(T^{2,0} \Sigma) \mid \text{$h$ is symmetric, $g$-traceless and $\delta_g h = 0$}} .
\]
An element of $S_{tt}(\Sigma,g)$ is usually called a \emph{tranverse traceless} tensor (with respect to the metric $g$). As shown in \cite{tromba2012teichmuller}, every element of $S_{tt}(\Sigma,g)$ can be written (uniquely) as the real part of a holomorphic quadratic differential $\Phi \in Q(\Sigma,[g])$, and vice versa for every $\Phi$, the tensor $\Re \Phi$ belongs to $S_{tt}(\Sigma,g)$. In particular, the space $S_{tt}(\Sigma,g)$ depends only on the conformal class of the metric $g$. If $g$ is a hyperbolic metric, then $S_{tt}(\Sigma,g)$ is tangent to the space $\altmathcal{M}^{-1}$ of hyperbolic metrics on $\Sigma$, and it is transverse to the orbit of $g$ by the action of the group of diffeomorphisms isotopic to the identity. Therefore, the tangent space of the Teichm\"uller space at the isotopy class of $g$ can be identified with $S_{tt}(\Sigma,g)$.

For any open set $\Omega \subseteq \Sigma$ and any $p \in [1,\infty)$, the Fischer-Tromba $p$-\hsk norm of $h \in S_{tt}(\Sigma,g)$ is defined as
\[
\norm{h}_{FT, L^p(\Omega)} \defin \left( \int_\Omega \norm{h}^p_g \dvol_g \right)^{1/p} ,
\]
where $\dvol_g$ is the area form induced by $g$. When $p = \infty$, we set $\norm{h}_{FT, L^\infty(\Omega)} \defin \sup_\Omega \norm{h}_g$. If $\Omega = \Sigma$, we simply write $\norm{\cdot}_{FT,p}$.

Let now $m$ be a point of the Teichm\"uller space, and let $g$ be a hyperbolic metric in the equivalence class $m$, with associated Riemann surface structure $X$.

\begin{lemma} \label{lemma:relations_between_norms}
The vector spaces $B_h(X)$ and $S_{tt}(\Sigma,g)$ are identified to $T_m \Teich(\Sigma)$ through the linear isomorphism
\[
\begin{matrix}
B_h(X) & \longrightarrow & S_{tt}(\Sigma,g) \\
\nu_\Phi & \longmapsto & 2 \Re \Phi .
\end{matrix}
\]
Moreover, for every $\Phi \in Q(X)$ we have
\[
\norm{\nu_\Phi}_{B,q} = \frac{1}{2 \sqrt{2}} \norm{2 \Re \Phi}_{FT,q} .
\]

\begin{proof}
Let $g_t = \rho_t \abs{\dd{z_t}}^2$ be a smooth $1$-\hsk parameter family of Riemannian metrics on $\Sigma$, with $g_0 = g$, and let $\Phi = \phi \dd{z_0}^2$ be a holomorphic quadratic differential on the Riemann surface $X_0 = (\Sigma,[g_0])$. If we require the identity map $(\Sigma,g_0) \rightarrow (\Sigma,g_t)$ to be quasi-conformal with harmonic Beltrami differential 
\[
\nu_{t \Phi}^0 \defin \frac{t \bar{\phi}}{\rho_0} \ \partial_{z_0} \otimes \dd{\bar{z}_0} ,
\]
then the Riemannian metric $g_t$ can be expressed as
\[
g_t = \rho_t \abs{\pdv{z_t}{z_0}}^2 \abs{\dd{z_0}}^2 + 2 t \rho_t \abs{\pdv{z_t}{z_0}}^2 \Re\left( \frac{\phi}{\rho_0} \dd{z_0}^2 \right) + O(t^2) .
\]
Therefore the first order variation of $g_t$ at $t = 0$ coincides with
\[
\dot{g}_0 = \left( \dv{t}\left. \rho_t \abs{\pdv{z_t}{z_0}}^2 \right|_{t = 0} \abs{\dd{z_0}}^2 \right) + 2 \Re \Phi  .
\]
The quantity $\dot{g}_0$ identifies with a tangent vector to the space $\altmathcal{M}$ of Riemannian metrics over $\Sigma$ at the point $g_0$. The first term in the expression above is conformal to the Riemannian metric $g_0$, hence it is tangent to the conformal class $[g_0] \subset \altmathcal{M}$. The remaining term $2 \Re \Phi$ is a symmetric, $g_0$-\hsk traceless and divergence-free tensor, so it lies in the subspace $S_{tt}(\Sigma, g_0)$ of $T_{g_0} \altmathcal{M}$. 

The computation above proves that the harmonic Beltrami differential $\nu_\Phi$, seen as an element of $T_m \Teich^\conf(\Sigma)$, corresponds to $2 \Re \Phi \in S_{tt}(\Sigma,g_0) \cong T_m \Teich^\hyp(\Sigma)$. Finally, an explicit computation shows the relation between the norms $\norm{\cdot}_{B,q}$ and $\norm{\cdot}_{FT,q}$.
\end{proof}
\end{lemma}

\section{A bound of the differential of the length} \label{section:bound_differential_length}

Let $\MesLam(\Sigma)$ denote the space of measured laminations of $\Sigma$. The aim of this section is to produce, given $\mu \in \MesLam(\Sigma)$, a quantitative upper bound of the differential of the length function $\mappa{L_\mu}{\Teich^\hyp(\Sigma)}{\R}$, which associates to every class of hyperbolic metrics $m \in \Teich^\hyp(\Sigma)$ the length of the $m$-\hsk geodesic realization of $\mu$. This estimate is the content of Proposition \ref{prop:bound_differential_length}, which will be our main technical ingredient to produce the upper bound of the dual volume in terms of the Weil-Petersson distance between the hyperbolic metrics on the convex core of a quasi-\hsk Fuchsian manifold.

We briefly sketch the structure of this section: Lemma \ref{lem:derivative_length_integral} describes a natural way to express the differential of $L_\mu$ applied to a first order variation of hyperbolic metrics $\dot{g}$. Lemma \ref{lem:ptwise_integral_bound} uses the properties of holomorphic functions to bound the pointwise value of a holomorphic quadratic differential at $x \in \Sigma$ with its $L^q$-\hsk norm on the ball centered at $x$. Then Proposition \ref{prop:bound_differential_length} will follow by selecting a first order variation $\dot{g}$ in $S_{tt}(\Sigma,g)$ and then carefully applying the bound of Lemma \ref{lem:ptwise_integral_bound} in the expression found in Lemma \ref{lem:derivative_length_integral}.

\vspace{0.5cm}

Let $m \in \Teich^\hyp(\Sigma)$ and $\mu \in \MesLam(\Sigma)$. Given a hyperbolic metric $g$ in the equivalence class $m$, we identify the measured lamination $\mu$ with its $g$-\hsk geodesic realization inside $(\Sigma,g)$. If $\lambda$ is a $g$-\hsk geodesic lamination of $\Sigma$ containing the support of $\mu$, we can cover $\lambda$ by finitely many flow boxes $\mappa{\sigma_j}{I \times I}{B_j}$, where $I = [0,1]$ and $\sigma_j$ is a homeomorphism verifying $\sigma_j^{-1}(\lambda) = D_j \times I$, for some closed subset $D_j$ of $I$. We select also a collection $\set{\xi_j}_j$ of smooth functions with supports contained in the interior of $B_j$ for every $j$, and such that $\sum_j \xi_j = 1$ over $\lambda$. Since the arcs $\sigma_j(I \times \set{s})$ are transverse to $\lambda$, it makes sense to integrate the first component of $\sigma_j$ with respect to the measure $\mu$. We set the \emph{length of $\mu$ with respect to $m$} to be the quantity
\[
L_\mu(m) \defin \sum_j \int_{D_j} \int_0^1 \xi_j(\sigma_j(p,\cdot)) \dd{\length(\cdot)} \dd{\mu}(p) ,
\]
where $\dd{\length(s)} = \norm{\partial_s \sigma_j(p,s)}_g \dd{s}$. More generally, given a measurable function $f$ defined on a neighborhood of $\lambda$, we define
\[
\iint_\lambda f \dd{\length} \dd{\mu} \defin \sum_j \int_{D_j} \int_0^1 \xi_j(\sigma_j(p,\cdot)) f(\sigma_j(p,\cdot)) \dd{\length(\cdot)} \dd{\mu}(p) .
\]
The quantity $L_\mu(m)$ does not depend on the choices we made of $\sigma_j$, $\xi_j$ and the hyperbolic metric $g$ in the equivalence class $m \in \Teich(\Sigma)$ (see e. g. \cite{bonahon1996shearing}). Therefore, any measured lamination $\mu$ of $\Sigma$ determines a positive function $L_\mu$ on the Teichm\"uller space $\Teich(\Sigma)$, which associates to any $m \in \Teich^\hyp(\Sigma)$ the length of the geodesic realization of $\mu$ in $m$.

Similarly, if $(g_t)_t$ is a smooth $1$-\hsk parameter family of hyperbolic metrics on $\Sigma$, with $g_0 = g$ and $\dot{g}_0 = \dot{g}$, we set
\[
\iint_\lambda \dd{\dot{\length}} \dd{\mu} \defin \frac{1}{2} \sum_j \int_{D_j} \int_0^1 \xi_j(\sigma_j(p,\cdot)) \, \frac{\dot{g} \left( \partial_s \sigma_j(p,\cdot), \partial_s \sigma_j(p,\cdot) \right)}{g \left( \partial_s \sigma_j(p,\cdot), \partial_s \sigma_j(p,\cdot) \right)} \dd{\length(\cdot)} \dd{\mu}(p) .
\]

\begin{lemma} \label{lem:derivative_length_integral}
Let $\mu$ be a measured lamination of $\Sigma$, and let $(m_t)_t$ be a smooth path in $\Teich^\hyp(\Sigma)$ verifying $m_0 = m$ and $\dot{m}_0 = \dot{m} \in T_m \Teich^\hyp(\Sigma)$. Then we have
\[
\dd(L_\mu)_m(\dot{m}) = \iint_\lambda \dd{\dot{\length}} \dd{\mu} ,
\]
where $\iint_\lambda \dd{\dot{\length}} \dd{\mu}$ is defined as above by selecting a smooth path $t \mapsto g_t$ of hyperbolic metrics representing $t \mapsto m_t$.
\begin{proof}
First we prove the statement when $\mu$ is a weight $1$ simple closed curve $\gamma$ in $\Sigma$. Let $\mappa{\gamma_t}{[0,1]}{\Sigma}$ denote a parametrization of the geodesic representative of $\gamma$ with respect to the hyperbolic metric $g_t$, which can be chosen to depend differentiably in $t$. Then the length of $\gamma_t$ with respect to the metric $g_t$ can be expressed as
\[
L_\gamma(m_t) = \int_0^1 \sqrt{g_t(\gamma_t'(s),\gamma_t'(s))} \dd{s} .
\]
Now, by taking the derivative of this expression in $t$ and using the fact that $\nabla \dot{\gamma}_0 \equiv 0$ (with $\nabla$ being the Levi-Civita connection of $g_0$), we obtain that
\[
\left. \dv{t} L_\gamma(m_t) \right|_{t = 0} = \frac{1}{2} \int_0^1 \frac{\dot{g}_0(\gamma_0'(s),\gamma_0'(s))}{\sqrt{g_0(\gamma_0'(s),\gamma_0'(s))}} \dd{s} ,
\]
which coincides with the quantity $\iint_\gamma \dd{\dot{\length}} \dd{\mu}$. By linearity we deduce the statement for any rational lamination $\mu = \sum_i a_i \gamma_i$.

Now, if $\mu$ is a general measured lamination, we select a sequence of rational laminations $(\mu_n)_n$ converging to $\mu$. As shown in \cite{kerckhoff1985earthquakes}, the functions $L_{\mu_n}$ are real analytic over $\Teich^\hyp(\Sigma)$ and they converge in the $\mathscr{C}^\infty$-\hsk topology on compact sets to $L_\mu$. In particular the terms $\dd(L_{\mu_n})_m(\dot{m})$ converge to $\dd(L_\mu)_m(\dot{m})$. On the other hand, it is simple to check that the expression $\iint_\lambda \dd{\dot{\length}} \dd{\mu}$ varies continuously in the measured lamination $\mu \in \MesLam(\Sigma)$ (see for instance the proof of \cite[Proposition~3.3]{mazzoli2021the_dual}, where an analogous result was proved for the realization of measured laminations inside a $1$-\hsk parameter family of convex co-compact hyperbolic $3$-\hsk manifolds). Hence the statement follows by an approximation argument.
\end{proof}
\end{lemma}

Before stating Lemma \ref{lem:ptwise_integral_bound}, we define for convenience the following quantities: for every $q \in [1,\infty)$ and $r > 0$, we set
\begin{equation} \label{eq:mult_const_pointwise}
C(r,q) \defin \left( \frac{2 q - 1}{4 \pi} \frac{(\cosh(r/2))^{4 q - 2}}{(\cosh(r/2))^{4 q - 2} - 1} \right)^{1/q} .
\end{equation}
When $q = \infty$, we define $C(r,\infty) \defin 1$ for every $r > 0$.

\begin{lemma} \label{lem:ptwise_integral_bound}
Let $(\Sigma,g)$ be a hyperbolic surface. Given $x \in \Sigma$ and $r < \injrad_g(x)$, we denote by $B_r(x)$ the metric ball of radius $r$ centered at $x \in \Sigma$. Then, for every $q \in [1,\infty]$ and for every holomorphic quadratic differential on $(\Sigma,[g])$, we have
\[
\norm{\Re \Phi_x} \leq C(r,q) \, \norm{\Re \Phi}_{FT,L^q(B_r(x))} .
\]
where $\norm{\Re \Phi_x}$ is the pointwise norm of the tensor $\Re \Phi$ at $x$.
\begin{proof} If $q = \infty$, the statement is clear. Consider $q < \infty$. By passing to the universal cover, we can assume the surface to be $\Delta = \set{z \in \C \mid \abs{z} < 1}$ and $x$ to be $0 \in \Delta$. The hyperbolic metric of $\Delta$ is of the form
\[
g_\Delta = \left( \frac{2}{1 - \abs{z}^2} \right)^2 \abs{\dd{z}}^2 ,
\]
where $z \in \Delta$ is the natural coordinate of $\Delta \subset \C$. In what follows, we will denote by $\norm{\cdot}$ the norm induced by the hyperbolic metric, and by $\norm{\cdot}_0$ the one induced by the standard Euclidean metric $\abs{\dd{z}}^2$. 

If $\Phi = \phi \dd{z}^2$ is a holomorphic quadratic representative, then for any $\rho \in (0,1)$ the residue theorem tells us that
\[
\phi(0) = \frac{1}{2 \pi i} \int_{\partial B^E_\rho} \frac{\phi(z)}{z} \dd{z} ,
\]
where $B^E_\rho = B^E_\rho(0) = \set{z \in \Delta \mid \abs{z} < \rho}$ (here $E$ stands for "Euclidean"). In particular we have
\begin{equation} \label{eq:jensen}
\abs{\phi(0)}^q \leq \left( \frac{1}{2 \pi} \int_0^{2 \pi} \abs*{\phi(\rho e^{i \theta})} \dd{\theta} \right)^q \leq \frac{1}{2 \pi} \int_0^{2 \pi} \abs*{\phi(\rho e^{i \theta})}^q \dd{\theta} ,
\end{equation}
where in the last step we used the H\"older inequality. At $z = \rho e^{i \theta}$, the hyperbolic norm of $\Re \Phi(z)$ can be expressed as follows:
\[
\norm{\Re \Phi(z)} = \frac{1}{\sqrt{2}} \abs*{\phi(\rho e^{i \theta})} \left( \frac{2}{1 - \rho^2} \right)^{- 2} \norm{\dd{z}^2}_0 .
\]
It is easy to check that the metric ball $B_r$ centered at $0$ with respect to the hyperbolic distance coincides with $B^E_{\tanh(r/2)}$, and that the hyperbolic volume form $\dvol$ is given by $\rho (2/(1 - \rho^2))^2 \dd{\rho} \dd{\theta}$. Combining all these facts, if we multiply the inequality \eqref{eq:jensen} by $\rho (2/(1 - \rho^2))^{2 - 2q}$ and we integrate in $\int_0^{\tanh r/2} \dd{\rho}$, we deduce that
\begin{align*}
\int_{B_r} \norm{\Re \Phi}^q \dvol & = 2^{- q/2} \norm{\dd{z}^2}^q_0 \ \int_0^{\tanh r/2} \rho \left( \frac{2}{1 - \rho^2} \right)^{2 - 2q} \int_0^{2 \pi} \abs*{\phi(\rho e^{i \theta})}^q \dd{\theta} \dd{\rho} \\
& \geq 2 \pi \abs{\phi(0)}^q \ 2^{- q/2 - 2(q - 1)} \norm{\dd{z}^2}^q_0 \ \int_0^{\tanh r/2} \rho (1 - \rho^2)^{2(q - 1)} \dd{\rho} \\
& = 4 \pi \norm{\Re \Phi(0)}^q \ \frac{1}{2 q - 1} \left(1 - \frac{1}{(\cosh(r/2))^{4 q - 2}} \right) \\
& = C(r,q)^{- q} \norm{\Re \Phi(0)}^q ,
\end{align*}
which proves the assertion.
\end{proof}
\end{lemma}

\noindent We state here another useful fact we will use in the proof of Proposition \ref{prop:bound_differential_length}:

\begin{lemma} \label{lemma:Fubini}
Let $(\Sigma,g)$ be a hyperbolic surface and let $\mu$ be a measured lamination on $\Sigma$. Then, for every $L^1$-\hsk function $\mappa{f}{N_r(\mu)}{\R}$ defined on the $r$-\hsk neighborhood of $\mu$ in $\Sigma$, we have
\[
\iint_{\lambda} \left( \int_{B_r(\cdot)} f \dvol_g \right) \dd{\length} \dd{\mu} = \int_\Sigma \left( \iint_{\lambda \cap B_r(\cdot)} \dd{\length} \dd{\mu} \right) f \dvol_g .
\]
\begin{proof}
Assume that $\mu$ is a $1$-\hsk weighted simple closed curve $\mappa{\gamma}{[0,1]}{\Sigma}$, and let $\tilde{f}$ denote the extension of the function $f$ to $\Sigma$ verifying $\tilde{f}(x) = 0$ for all $x \in \Sigma \setminus N_r(\gamma)$. We set $\mappa{\xi}{\Sigma \times \Sigma}{\R}$ to be the function taking value $\xi(x,y) = 1$ if the distance between $x$ and $y$ is less than $r$, and $\xi(x,y) = 0$ otherwise. Then the integral on the left can be expressed as
\[
\int_0^1 \int_\Sigma \tilde{f}(x) \ \xi(x,\gamma(t)) \dvol_g(x) \dd{\length(t)} .
\]
Applying Fubini's theorem we obtain
\begin{align*}
\int_0^1 \int_\Sigma \tilde{f}(x) \xi(x,\gamma(t)) \dvol_g(x) \dd{\length(t)} & = \int_\Sigma \int_0^1  \xi(x,\gamma(t)) \dd{\length(t)} \tilde{f}(x) \dvol_g(x) \\
& = \int_\Sigma \left( \int_{\gamma^{-1}(B_r(x))} \dd{\length(t)} \right) \tilde{f}(x) \dvol_g(x) .
\end{align*}
The last term coincides with the right term of the equality in the statement in the case $\mu = \gamma$. By linearity we deduce the statement when $\mu$ a rational lamination, and by continuity of the two integrals in the statement with respect with $\mu$ we obtain the result for any general measured lamination.
\end{proof}
\end{lemma}

Let $m \in \Teich^\hyp(\Sigma)$ and $\mu \in \MesLam(\Sigma)$, and select a hyperbolic metric $g$ in the equivalence class $m$. If $(\widetilde{\Sigma},\tilde{g})$ denotes the universal cover of $(\Sigma, g)$, we define
\[
D(m,\mu,r) \defin \sup_{\tilde{x} \in \widetilde{\Sigma}} \iint_{\tilde{\lambda} \cap B_r(\tilde{x})} \dd{\tilde{\length}} \dd{\tilde{\mu}} < \infty .
\]
where $\tilde{\lambda}$ denotes the support of the measured lamination $\tilde{\mu}$. In other words, $D(m,\mu,r)$ is the supremum, over the points $\tilde{x}$ in the universal cover $\widetilde{\Sigma}$, of the length of the portion of $\tilde{\mu}$ contained in the ball centered at $\tilde{x}$ of radius $r$.

\begin{proposition} \label{prop:bound_differential_length}
For any $r > 0$ and for any $p \in [1,\infty]$ we have
\[
\abs{\dd(L_\mu)_m(\dot{m})} \leq L_\mu(m)^{1/p} \ C(r,q) \ D(m,\mu,r)^{1/q} \norm{\nu}_{B,q} ,
\]
where $p$ and $q$ are conjugate exponents, i. e. $\frac{1}{p} + \frac{1}{q} = 1$, and $\nu$ denotes the harmonic Beltrami differential representing the tangent direction $\dot{m} \in T_m \Teich^\hyp(\Sigma)$. In particular, for $p = 2$, we have
\[
\norm{\dd(L_\mu)_m}_{Q,2} \leq C(r,2) \sqrt{L_\mu(m) \ D(m,\mu,r)} .
\]
\begin{proof}
As described in \cite{tromba2012teichmuller}, there exists a unique symmetric transverse-\hsk traceless tensor $\varphi \in S_{tt}(\Sigma,g)$ representing the tangent vector $\dot{m} \in T_m \Teich^\hyp(\Sigma)$, which is of the form $\Re \Phi = \varphi$ for some holomorphic quadratic differential $\Phi$ on $(\Sigma,[g])$. Observe that, since $\varphi$ lies in $S_{tt}(\Sigma,g)$, we have $\abs{\varphi(v,v)} \leq \frac{1}{\sqrt{2}} \norm{\varphi}_g \norm{v}_g^2$ for every tangent vector $v$. Making use of Lemma \ref{lem:derivative_length_integral} and recalling the definition of the term $\iint_\lambda \dd{\dot{\length}} \dd{\mu}$ we see that
\[
\abs{\dd(L_\mu)_m(\dot{m})} = \abs{\iint_\lambda \dd{\dot{\length}} \dd{\mu}} \leq \frac{1}{2 \sqrt{2}} \iint_\lambda \norm{\varphi}_g \dd{\length} \dd{\mu} .
\]
By applying the H\"older inequality on the right-side integral, we get
\begin{equation} \label{eq:holder_ineq}
\abs{\dd(L_\mu)_m(\dot{m})} \leq \frac{1}{2 \sqrt{2}} \iint_\lambda \norm{\varphi}_g \dd{\length} \dd{\mu} \leq \frac{L_\mu(m)^{1/p}}{2 \sqrt{2}} \left( \iint_\lambda \norm{\varphi}^q_g \dd{\length} \dd{\mu} \right)^{1/q} .
\end{equation}
Now we estimate the integral $\iint_\lambda \norm{\varphi}^q_g \dd{\length} \dd{\mu}$ by lifting it to a suitable covering of $\Sigma$, and then applying Lemma \ref{lem:ptwise_integral_bound}. More precisely, let $(\widehat{\Sigma},\hat{g}) \rightarrow (\Sigma,g)$ be a $N$-\hsk index covering so that $\injrad(\widehat{\Sigma},\hat{g}) > r$, for some $N \in \N$. We denote by $\hat{\bullet}$ the lift of the object $\bullet$ on $\widehat{\Sigma}$. It is immediate to check that the following relation holds
\[
\iint_\lambda \norm{\varphi}^q_g \dd{\length} \dd{\mu} = \frac{1}{N} \iint_{\hat{\lambda}} \norm{\hat{\varphi}}_{\hat{g}}^q \dd{\hat{\length}} \dd{\hat{\mu}} .
\]
Then, by applying Lemma \ref{lem:ptwise_integral_bound} on the surface $(\widehat{\Sigma},\hat{g})$ and at each point $\hat{x} \in \hat{\lambda}$, we get
\begin{align*}
\iint_\lambda \norm{\varphi}^q_g \dd{\length} \dd{\mu} & = \frac{1}{N} \iint_{\hat{\lambda}} \norm{\hat{\varphi}}_{\hat{g}}^q \dd{\hat{\length}} \dd{\hat{\mu}} \\
& \leq \frac{C(r,q)^q}{N} \iint_{\hat{\lambda}} \norm{\hat{\varphi}}_{FT, L^q(B_r(\cdot))}^q \dd{\hat{\length}} \dd{\hat{\mu}} \\
& = \frac{C(r,q)^q}{N} \iint_{\hat{\lambda}} \left( \int_{B_r(\cdot)} \norm{\hat{\varphi}}^q_{\hat{g}} \dvol_{\hat{g}} \right) \dd{\hat{\length}} \dd{\hat{\mu}} .
\end{align*}
Using Lemma \ref{lemma:Fubini} and the definition of $D(m,\mu,r)$, we obtain
\begin{align*}
\iint_{\hat{\lambda}} \left( \int_{B_r(\cdot)} \norm{\hat{\varphi}}^q_{\hat{g}} \dvol_{\hat{g}} \right) \dd{\hat{\length}} \dd{\hat{\mu}} & = \int_{\widehat{\Sigma}} \left( \iint_{\hat{\lambda} \cap B_r(\cdot)} \dd{\hat{\length}} \dd{\hat{\mu}} \right) \norm{\hat{\varphi}}^q_{\hat{g}} \dvol_{\hat{g}} \\
& \leq D(m,\mu,r) \int_{\widehat{\Sigma}} \norm{\hat{\varphi}}^q_{\hat{g}} \dvol_{\hat{g}} \\
& = N \ D(m,\mu,r) \ \norm{\varphi}^q_{FT,q} ,
\end{align*}
where, in the last step, we are using again the fact that $(\Sigma,g) \rightarrow (\widehat{\Sigma},\hat{g})$ is a $N$-\hsk index covering. Combining the last two estimates, we obtain
\begin{equation} \label{eq:bound_with_Lp_norm}
\iint_\lambda \norm{\varphi}^q_g \dd{\length} \dd{\mu} \leq C(r,q)^q \ D(m,\mu,r) \ \norm{\varphi}^q_{FT,q} .
\end{equation}
Using the inequalities \eqref{eq:holder_ineq} and \eqref{eq:bound_with_Lp_norm}, we have shown that
\[
\abs{\dd(L_\mu)_m(\dot{m})} \leq \frac{L_\mu(m)^{1/p} \ C(r,q) \ D(m,\mu,r)^{1/q}}{2 \sqrt{2}} \ \norm{\varphi}_{FT,q} .
\]
Finally, by applying Lemma \ref{lemma:relations_between_norms}, we obtain
\[
\abs{\dd(L_\mu)_m(\dot{m})} \leq L_\mu(m)^{1/p} \ C(r,q) \ D(m,\mu,r)^{1/q} \norm{\nu}_{B,q} .
\]
The last assertion follows from the estimate we just proved for $p = 2$ and from Lemma \ref{lem:L2_dual_norm}.
\end{proof}
\end{proposition}

\section{The differential of the dual volume} \label{section:dual_volume}

In this section we use Proposition \ref{prop:bound_differential_length} to bound the differential of the function $V_\CC^*$, which associates to each quasi-\hsk Fuchsian manifold $M$ the dual volume of its convex core. The link between $V_\CC^*$ and the differential of the length of the bending measured lamination is given by the dual Bonahon-\hsk Schl\"afli formula (see \cite{krasnov2009symplectic}, \cite{mazzoli2021the_dual}). We will recall and make use of the results by \citet{bridgeman_brock_bromberg2017}, and \citet{yarmola2016improved}, which will allow us to estimate uniformly the quantities $L_\mu(m)$ and $D(m,\mu,r)$ appearing in Proposition \ref{prop:bound_differential_length}.

\vspace{0.5cm}

Let $M$ be a complete $3$-\hsk dimensional hyperbolic manifold, We say that a subset $C \subset M$ is convex if for any geodesic arc $\gamma$ of $M$ connecting two points $x$ and $y$ of $C$ (possibly equal), the arc $\gamma$ is fully contained in $C$. The manifold $M$ is called \emph{quasi-\hsk Fuchsian} if it is homeomorphic to $\Sigma \times \R$ and it contains a compact non-empty convex subset. In this case, the intersection of all non-\hsk empty compact convex subsets of $M$ is a non-empty compact convex subset, called the \emph{convex core} of $M$ and denoted by $\CC M$, which is obviously minimal with respect to the inclusion.

The boundary of the convex core $\CC M$ is almost everywhere totally geodesic and it is homeomorphic to two copies of $\Sigma$. If we identify the universal cover of $M$ with $\Hyp^3$, then the preimage of $\partial \CC M$ in $\Hyp^3$ is the union of two locally convex pleated planes $H^\pm$ bent along a measured lamination $\tilde{\mu}$. Since these pleated planes are invariant under the action of the fundamental group of $M$, they determine two hyperbolic metrics $m^+$, $m^- \in \Teich^\hyp(\Sigma)$ and two measured laminations $\mu^+$, $\mu^- \in \MesLam(\Sigma)$. We will denote the couple of metrics $(m^+,m^-)$ by $m \in \Teich(\partial \CC M)$ and the pairs of measured laminations $(\mu^+,\mu^-)$ by $\mu \in \MesLam(\partial \CC M)$.

The action of the fundamental group $\Gamma$ of $M$ naturally extends to $\partial \Hyp^3 \cong \C \Proj^1$ by M\"obius tranformations. Given any $x_0 \in \Hyp^3$, the subset $\Lambda$ of accumulation points of $\Gamma x_0$ in $\partial \Hyp^3$ is called the \emph{limit set} of $\Gamma$.  The action of $\Gamma$ is free and properly discontinuous on $\partial \Hyp^3 \setminus \Lambda$, and it determines a pair of Riemann surface structures $c^+$, $c^-$ on $\Sigma$ and $\overline{\Sigma}$ (the surface $\Sigma$ endowed with the opposite orientation), called the \emph{conformal structures at infinity} of $M$. A well-known result of \cite{bers1960simultaneous} states that the space of quasi-\hsk Fuchsian structures on $\Sigma \times \R$ is parametrized by the couple of conformal structures at infinity. In other words, the map
\[
\begin{matrix}
B \vcentcolon & \QF(\Sigma) & \longrightarrow & \Teich^\conf(\Sigma) \times \Teich^\conf(\overline{\Sigma}) \\
& M & \longmapsto & (c^+,c^-)
\end{matrix}
\]
is a homeomorphism. In fact $B$ is a biholomorphism if we endow $\QF(\Sigma)$ with the complex structure of subset of the character variety $\chi(\pi_1 \Sigma,\Proj \SL_2 \C)$, and the natural complex structure of $\Teich^\conf(\Sigma)$. Another natural map on $\QF(\Sigma)$ is
\[
\begin{matrix}
\Psi \vcentcolon & \QF(\Sigma) & \longrightarrow & \Teich^\hyp(\Sigma) \times \Teich^\hyp(\Sigma) \\
& M & \longmapsto & (m^+,m^-) ,
\end{matrix}
\]
which Thurston conjectured to be another parametrization of the space of quasi-\hsk Fuchsian manifolds. \citet{bonahon1998variations} proved that the map $\Psi$ is $\mathscr{C}^1$ (and actually not $\mathscr{C}^2$), therefore a first order variation of quasi-\hsk Fuchsian structures $\dot{M}$ determines a first order variation of the induced hyperbolic structures $\dot{m}$ on the convex core.

\begin{definition}
Let $N \subset M$ be a compact convex subset of $M$ with regular boundary. The \emph{dual volume} of $N$ is defined as
\[
\Vol^*(N) \defin \Vol(N) - \frac{1}{2} \int_{\partial N} H \dd{a} ,
\]
where $H$ is trace of the shape operator of $\partial N$ with respect to the normal vector pointing towards the interior of $N$. The dual volume of the convex core of $M$ is set to be:
\[
V_\CC^*(M) \defin \Vol(\CC M) - \frac{1}{2} L_\mu(m) .
\]
\end{definition}

Contrary to the usual hyperbolic volume of the convex core (see \citep{bonahon1998schlafli} for details), the dual volume $V_\CC^*(M)$ turns out to be a $\mathscr{C}^1$-\hsk function on the space of quasi-\hsk Fuchsian manifolds, and its variation is described by the following result: 

\begin{theorem}[\cite{krasnov2009symplectic}, \cite{mazzoli2021the_dual}] \label{thm:dual_bonahon_schlafli}
Let $(M_t)_{t \in (- \varepsilon,\varepsilon)}$ be a smooth $1$-\hsk parameter family of quasi-\hsk Fuchsian structures. We denote by $\mu = \mu_0 \in \MesLam(\partial \CC M)$ the bending measure of the convex core of $M = M_0$ and by $(m_t)_t$ the family of hyperbolic metrics on the boundary of the convex core $\CC M_t$. Then the derivative of the dual volume of $\CC M_t$ exists and it verifies
\[
\dd{V_\CC^*}(\dot{M}) = - \frac{1}{2} \dd(L_{\mu})_{m}(\dot{m}) ,
\]
where $m = m_0$ and $\dot{m} = \dot{m}_0 \in T_m \Teich^\hyp(\Sigma)$.
\end{theorem}

This fact has been initially proved by \cite{krasnov2009symplectic} making use of Bonahon's work on the variation of the hyperbolic volume in \citep{bonahon1998schlafli}. The author of this paper has recently described an alternative proof of this relation that does not require the results of \citep{bonahon1998schlafli}, which can be found in \cite{mazzoli2021the_dual}.

An immediate corollary of the variation formula of the dual volume and of our estimate in Proposition \ref{prop:bound_differential_length} is the following:

\begin{proposition} \label{prop:diff_dual_volume_intermediate}
Let $\mappa{V_\CC^*}{\QF(\Sigma)}{\R}$ denote the function associating to each quasi-\hsk Fuchsian manifold $M$ the dual volume of its convex core $\CC M$. Then for every $r > 0$ and for every $p \in [1,\infty]$ we have
\[
\abs{\dd{V_\CC^*}(\dot{M})} \leq \frac{1}{2} \ L_\mu(m)^{1/p} \ C(r,q) \ D(m,\mu,r)^{1/q} \ \norm{\nu}_{B,q} ,
\]
where $C(r,q)$ and $D(m,\mu,r)$ are the constants defined in the previous section, $p$ and $q$ are conjugated exponents, and $\nu$ denotes the harmonic Beltrami differential representing the variation of the hyperbolic metric of the convex core.
\end{proposition}

In the remaining part of this section we describe a procedure to obtain a multiplicative factor in the above statement depending only on $p$ and the genus of $\Sigma$.

As we mentioned before, the lift of the boundary of the convex core is the union of two locally bent pleated planes $H^\pm$, which are embedded in $\Hyp^3$. This property turns out to determine uniform upper bounds of the quantities $L_\mu(m)$ and $D(m,\mu,r)$ appearing in the statement of Proposition \ref{prop:bound_differential_length}. The first results in this direction have been developed by Epstein and Marden in \cite{epstein2006fundamentals}. In our exposition, we will recall and make use of the works of \citet{bridgeman_brock_bromberg2017} and \citet{yarmola2016improved}, which will give us separate bounds for $L_\mu(m)$ and $D(m,\mu,r)$, respectively. We will also require $r$ to be less than $\ln(3)/2$. This restriction simplifies our argument in the proof of Corollary \ref{cor:bound_total_length}. However, we do not exclude the possibility that a joint study of the quantity $L_\mu(m)^{1/p} D(m,\mu,r)^{1/q}$ and a careful choice of $r$ might improve the multiplicative constants obtained here.

First we focus on the term $D(m,\mu,r)$, which we defined before the statement of Proposition \ref{prop:bound_differential_length}. Let $\tilde{\lambda}$ denote the geodesic lamination in $\widetilde{\Sigma}$ given by the lift of the support of the measured lamination $\mu$. Let $Q$ be a component of $\widetilde{\Sigma} \setminus \tilde{\lambda}$ and let $l_1$, $l_2$, $l_3$ be three boundary components of $Q$. We will use the following fact:

\begin{lemma}[{\cite[Corollary~II.2.4.3]{epstein2006fundamentals}}] \label{lemma:distance_boundaries}
Let $r < \ln(3)/2 = \arcsinh(1/ \sqrt{3})$, and suppose we have a point $x \in Q$ 
which is at distance $\leq \arcsinh(e^{- r})$ from both $l_2$ and $l_3$. Then its distance from $l_1$ is $> r$.
\end{lemma}

Following \cite{yarmola2016improved}, given $\tilde{\mu}$ a measured lamination on $\Hyp^2$, we denote by $\norm{\tilde{\mu}}_s$ the supremum over $\alpha$ of the transverse measure of $\tilde{\mu}$ along $\alpha$, where $\alpha$ varies among the geodesic arcs in $\Hyp^3$ of length $s > 0$ which are transverse to the support of $\tilde{\mu}$.

\begin{theorem}[\cite{yarmola2016improved}] \label{thm:yarmola_bound_measure}
Let $s \in (0,2 \arcsinh 1)$ and let $\tilde{\mu}$ be a measured lamination of $\Hyp^2$ so that the pleated plane with bending measure $\tilde{\mu}$ is embedded inside $\Hyp^3$. Then
\[
\norm{\tilde{\mu}}_s \leq 2 \arccos\left( - \sinh(s/2) \right) .
\]
\end{theorem}

\begin{corollary} \label{cor:bound_total_length}
Let $\mu \in \MesLam(\Sigma)$ and $m \in \Teich^\hyp(\Sigma)$ be the bending measure and the hyperbolic metric, respectively, of the boundary of an incompressible hyperbolic end inside a hyperbolic convex co-compact $3$-\hsk manifold. Then for every $r < \ln(3)/2$ we have 
\[
D(m,\mu,r) \leq 4 r \arccos\left( - \sinh r \right) .
\]
Moreover, for every $\varepsilon > 0$ there exists $m_\varepsilon \in \Teich^\hyp(\Sigma)$ and $\mu_\varepsilon \in \MesLam(\Sigma)$ as above verifying
\[
D(m_\varepsilon,\mu_\varepsilon,r) \geq 2 (\pi - \varepsilon) r \qquad \forall r > 0 .
\]
\begin{proof}
Let $g$ be a hyperbolic metric in the equivalence class $m \in \Teich^\hyp(\Sigma)$. We denote by $(\widetilde{\Sigma},\tilde{g}) \rightarrow (\Sigma,g)$ the Riemannian universal cover of $(\Sigma,g)$ and by $\tilde{\lambda}$ the support of the lift $\tilde{\mu}$ of the measured lamination $\mu$ to $\widetilde{\Sigma}$. Given a point $\tilde{x}$ in $\widetilde{\Sigma}$ and a positive $r < \ln(3)/2$, we are looking for an upper bound of the length of $\tilde{\mu} \cap B_r(\tilde{x})$, where $B_r(\tilde{x})$ denotes the metric ball of radius $r$ at $\tilde{x}$.

The convenience of considering $r < \ln(3)/2$ comes from Lemma \ref{lemma:distance_boundaries}: under this hypothesis, any complementary region $Q$ of the geodesic lamination $\tilde{\lambda}$ at distance less than $r$ from $x$ has at most two components of its boundary intersecting $B_r(x)$. A simple argument proves that, if this happens, we can find a geodesic path $\alpha$ of length $< 2 r$ that intersects all the leaves of $\tilde{\lambda} \cap B_r(\tilde{x})$. Each leaf of $\tilde{\lambda} \cap B_r(\tilde{x})$ has length $< 2 r$, therefore the length of $\tilde{\mu} \cap B_r(\tilde{x})$ is bounded by $2 r$ (the length of each leaf) times the total mass $\tilde{\mu}(\alpha)$, which can be estimated applying Theorem \ref{thm:yarmola_bound_measure} with $s = 2r < \ln 3 < 2 \arcsinh 1$. This proves the first part of the statement\footnote{See also Remark \ref{rmk:improve_multipl_const}.}.

For what concerns the last part of the assertion, we fix a simple closed curve $\gamma$ and we assign it the weight $\pi - \varepsilon$. By the work of \citet{bonahon_otal2004laminations}, we can find a quasi-\hsk Fuchsian manifold $M_\varepsilon$ realizing $(\pi - \varepsilon) \gamma$ as the bending lamination of the upper component of the boundary of the convex core $\partial^+ \CC M_\varepsilon$. It is immediate to check that, if $m_\varepsilon$ is the hyperbolic metric of $\partial^+ \CC M_\varepsilon$, then $D(m_\varepsilon,\mu_\varepsilon,r) \geq 2 (\pi - \varepsilon) r$ for all $r > 0$.
\end{proof}
\end{corollary}

For the bound of the term $L_\mu(m)$, we will apply the following result:

\begin{theorem}[{\cite[Theorem~2.16]{bridgeman_brock_bromberg2017}}] \label{thm:bound_length}
Let $\mu \in \MesLam(\Sigma)$ and $m \in \Teich^\hyp(\Sigma)$ be the bending measure and the hyperbolic metric, respectively, of the boundary of an incompressible hyperbolic end inside a hyperbolic convex co-compact $3$-\hsk manifold. Then
\[
L_\mu(m) \leq 6 \pi \abs{\chi(\Sigma)} .
\]
\end{theorem}

\noindent Finally, given $p \in (1,\infty)$ and $r < \ln(3)/2$, we set
\begin{align*}
K(r,p) & \defin \frac{1}{2} (24 \pi)^{1/p} \ C(r,q) \ (4 r \arccos(- \sinh r))^{1/q} \\
& = \frac{1}{2} (24 \pi)^{1/p} \left( \frac{2 q - 1}{\pi} \frac{(\cosh(r/2))^{4 q - 2}}{(\cosh(r/2))^{4 q - 2} - 1} \ r \arccos(- \sinh r) \right)^{1/q} ,
\end{align*}
where $C(r,q)$ was defined in equation \ref{eq:mult_const_pointwise}. We define also
\[
K(r,1) = 12 \pi, \qquad K(r,\infty) = \frac{r \arccos(- \sinh r)}{2 \pi \tanh^2(r/2)} .
\]
\begin{corollary} \label{cor:bound_diff_dual_vol}
In the same notations of Proposition \ref{prop:diff_dual_volume_intermediate}, for every $p \in [1,\infty]$ we have
\[
\abs{\dd{V_\CC^*}(\dot{M})} \leq K(p) (g - 1)^{1/p} \norm{\nu}_{B,q} ,
\]
where $K(p) \defin K(\ln(3)/2,p)$ and $\nu$ denotes the harmonic Beltrami differential representing the variation of the hyperbolic metrics on the boundary of the convex core $\partial \CC M$ of $M$. We have in particular that $K(2) \approx 10.3887$.
\begin{proof}
We combine Proposition \ref{prop:diff_dual_volume_intermediate}, Corollary \ref{cor:bound_total_length} and Theorem \ref{thm:bound_length} on the upper and lower components of $\partial \CC M = \partial \CC M_0$, and then we take the limit as $r$ goes to $\ln(3)/2$.
\end{proof}
\end{corollary}

\noindent We can compare this statement with the analogous bound for the differential of the renormalized volume:
\begin{theorem}[{\cite{schlenker2013renormalized}}]
Let $\mappa{\RVol}{\QF(\Sigma)}{\R}$ denote the function associating to each quasi-\hsk Fuchsian manifold $M$ its renormalized volume. Then for every $p \in [1,\infty]$ we have
\[
\dd{\RVol}(\dot{M}) \leq H(p) (g - 1)^{1/p} \norm{\dot{c}}_{B,q} ,
\]
where $\dot{c}$ denotes the variation of the conformal structures at infinity of $M$, and where $H(p) \defin \frac{3}{2} (8 \pi)^{1/p}$.
\end{theorem}

\begin{remark} \label{rmk:improve_multipl_const}
From the first part of the proof of Corollary \ref{cor:bound_total_length} is clear that our estimate of the constant $D(m,\mu,r)$ is far from being optimal. However, using the second part of the assertion, it is easy to see that the possible improvement of the constant $K(2)$ is not enough to make the multiplicative constant in Theorem \ref{thm:bound_dual_volume_WP} to be less than $3 \sqrt{\pi}$, which is the one appearing in the analogous statement for the renormalized volume. Because of this, we preferred to present a simpler but rougher argument.
\end{remark}

\section{Dual volume and Weil-Petersson distance} \label{section:dual_volume_WP_distance}

This section is dedicated to the proof of the linear upper bound of the dual volume of a quasi-\hsk Fuchsian manifold $M$ in terms of the Weil-Petersson distance between the hyperbolic structures on the boundary of its convex core $\CC M$. As we mentioned in Section \ref{section:dual_volume}, the data of the hyperbolic metrics of $\partial \CC M$ is only conjectured to give a parametrization of the space of quasi-\hsk Fuchsian manifolds, contrary to what happens with the conformal structures at infinity. In particular, the same strategy used in \cite{schlenker2013renormalized} to bound the renormalized volume cannot be immediately applied. In order to overcome this problem, we will take advantage of the foliation by constant Gaussian curvature surfaces ($k$-\hsk surfaces) of $M \setminus \CC M$, whose existence has been proved by \citet{labourie1991probleme} (see also Remark \ref{rmk:duality_hyp_desitter}). The space of hyperbolic structures with strictly convex boundary on $\Sigma \times [0,1]$ is parametrized by the data of the metrics on its boundary, as proved in \cite{schlenker2006hyperbolic}. In particular, the Teichm\"uller classes of the metrics of the upper and lower $k$-surfaces parametrize the space of quasi-\hsk Fuchsian structures of topological type $\Sigma \times \R$. Moreover, the first order variation of the dual volume of the region $M_k$ encosed between the two $k$-\hsk surfaces is intimately related to the notion of landslide, which was first introduced and studied in \cite{bonsante2013a_cyclic}, \cite{bonsante2015a_cyclic}. This connection will be very useful to relate the first order variation of $V_\CC^*(M)$ and of $\Vol^*(M_k)$, as $k$ goes to $-1$, allowing us to prove Theorem \ref{thm:bound_dual_volume_WP} using an approximation argument, together with the bounds obtained in the previous Section.

\subsection{Constant Gaussian curvature surfaces}

\noindent The existence of the foliation by constant Gaussian curvature surfaces is guaranteed by the following result:
\begin{theorem}[{\cite[Th\'eor\`eme~2]{labourie1991probleme}}] \label{thm:k_surfaces_foliation}
Every geometrically finite $3$-\hsk dimensional hyperbolic end $E$ is foliated by a family of strictly convex surfaces $(\Sigma_k)_k$ with constant curvature $k \in (-1,0)$. As $k$ goes to $-1$, the surface $\Sigma_k$ converges to the locally convex pleated boundary of $E$, and as $k$ goes to $0$, $\Sigma_k$ approaches the conformal boundary at infinity $\partial_\infty E$.
\end{theorem}

A surface of constant Gaussian curvature $k$ embedded in some hyperbolic $3$-\hsk manifold is called a \emph{$k$-\hsk surface}. From the Gauss equations we see that the extrinsic curvature of a $k$-surface is equal to $k + 1$. Therefore, if $k$ is in $(-1,0)$, the principal curvatures have the same sign and never vanish. In particular the leaves of the foliation of Theorem \ref{thm:k_surfaces_foliation} are all convex surfaces.

Given a quasi-\hsk Fuchsian manifold $M$, we denote by $m_k^\pm(M) \in \Teich^\hyp(\Sigma)$ the isotopy classes of the hyperbolic metrics $- k \ \I_k^\pm$, where $\I_k^\pm$ is the first fundamental form of the upper/lower $k$-\hsk surface $\Sigma^\pm_k$ of $M$. Then for every $k \in (-1,0)$ we have maps
\[
\begin{matrix}
\Psi_k \vcentcolon & \QF(\Sigma) & \longrightarrow & \Teich^\hyp(\Sigma) \times \Teich^\hyp(\Sigma) \\
& M & \longmapsto & (m_k^-(M),m_k^+(M)) . \\
\end{matrix}
\]
The family of functions $(\Psi_k)_k$ is related to the maps $\Psi$ and $B$ we considered in Section \ref{section:dual_volume}. As $k$ goes to $-1$, $\Psi_k(M)$ converges to $\Psi(M)$, and as $k$ goes to $0$, $\Psi_k(M)$ converges to $B(M)$ (see \cite[Section~3]{labourie1992surfaces}). The convenience in considering the foliation by $k$-surfaces relies in the following result, based on the works of \citet{labourie1991probleme} and \citet{schlenker2006hyperbolic}:

\begin{theorem} \label{thm:ksurfaces_homeo}
The map $\Psi_k$ is a $\mathscr{C}^1$-\hsk diffeomorphism for every $k \in (-1,0)$.
\begin{proof}
Let $(N,\partial N)$ be a compact connected $3$-\hsk manifold admitting a hyperbolic structure with convex boundary. \citet{schlenker2006hyperbolic} proved that any Riemannian metric with Gaussian curvature $> - 1$ on $\partial N$ is uniquely realized as the restriction to the boundary of a hyperbolic metric on $N$ with smooth strictly convex boundary. In other words, if $\mathscr{G}$ and $\mathscr{H}$ denote the spaces of isotopy classes of metrics on $N$ with strictly convex boundary and of metrics on $\partial N$ with Gaussian curvature $> -1$, respectively, then the restriction map
\[
\begin{matrix}
r \vcentcolon & \mathscr{G} & \longrightarrow & \mathscr{H} \\
& [g] & \longmapsto & [g|_{\partial N}]
\end{matrix}
\]
is a homeomorphism. The surjectivity had already been established by Labourie in \cite{labourie1991probleme}, therefore the proof proceeds by showing the local injectivity of $r$. To do so, the strategy in \cite{schlenker2006hyperbolic} is to apply the Nash-Moser implicit function theorem.

Let us fix now a $k \in (-1,0)$, and consider $N = \Sigma \times I$. If $\mathscr{G}_k$ is the space of hyperbolic structures on $N$ whose boundary has constant Gaussian curvature equal to $k$, then $\mathscr{G}_k$ identifies with the space of quasi-\hsk Fuchsian manifolds $\QF(\Sigma)$, thanks to Theorem \ref{thm:k_surfaces_foliation} and the fact that any hyperbolic structure with convex boundary on $N$ uniquely extends to a quasi-\hsk Fuchsian structure (see e. g. \cite[Theorem~I.2.4.1]{canary2006fundamentals}). In addition, the space $\mathscr{H}_k$ of constant $k$ Gaussian curvature structures on $\partial N$ can be interpreted as the product of two copies of the Teichm\"uller space $\Teich^\hyp(\Sigma)$, one for each component of $\partial N$. Therefore the restriction of $r$ over $\mathscr{G}_k$ can be identified with $\Psi_k$. The map $\Psi_k$ is now a function between finite dimensional differential manifolds. The fact that $r$ verifies the hypotheses to apply the Nash-Moser inverse function theorem implies in particular that $\Psi_k$ verifies the hypotheses to apply the ordinary inverse function theorem between finite dimensional manifolds. In particular, this shows that $\Psi_k$ is a $\mathscr{C}^1$-\hsk diffeomorphism, for any $k \in (-1,0)$, as desired.
\end{proof}
\end{theorem}

\subsection{The proof of Theorem \ref{thm:bound_dual_volume_WP}}

In the following we outline the proof of Theorem \ref{thm:bound_dual_volume_WP}. For every $k \in (-1,0)$, we consider the function
\[
V_k^* \vcentcolon \QF(\Sigma) \longrightarrow \R ,
\]
which associates with a quasi-\hsk Fuchsian manifold $M'$ the dual volume of the convex subset enclosed by the two $k$-\hsk surfaces of $M'$. By Theorem \ref{thm:ksurfaces_homeo} and by the dual differential Schl\"afli formula (see Theorem \ref{thm:dual_diff_schlafli} below), the function $V_k^* \circ \Psi_k^{-1}$ is $\mathscr{C}^1$ for every $k \in (- 1, 0)$. 

Let now $M$ be a fixed quasi-\hsk Fuchsian manifold. Since the Teichm\"uller space endowed with the Weil-Petersson metric is a unique geodesic space \cite{wolpert1987geodesic}, there exists a unique Weil-Petersson geodesic $\mappa{\beta_k}{[0,1]}{\Teich^\hyp(\Sigma)}$ verifying $\beta_k(0) = m_k^-$ and $\beta_k(1) = m_k^+$, where $m^\pm_k = m^\pm_k(M)$ are the hyperbolic structures on the upper and lower $k$-\hsk surfaces of $M$. We set $\gamma_k$ to be the path in $\Teich^\hyp(\Sigma) \times \Teich^\hyp(\Sigma)$ given by $\gamma_k(t) \defin (\beta_k(t),m_k^-)$. By construction $(\Psi_k^{-1} \circ \gamma_k)_k$ are paths of quasi-\hsk Fuchsian structures such that $(\Psi^{-1}_k \circ \gamma_k)(0) = \Psi_k^{-1}(m_k^-,m_k^-)$ is Fuchsian for every $k \in (-1,0)$, and $(\Psi_k^{-1} \circ \gamma_k)(1) = \Psi_k^{-1}(m_k^+,m_k^-)$ is equal to $M$. We now introduce the following notation: if $\varphi$ is a cotangent vector at a point $(p,q)$ of $\Teich^\hyp(\Sigma) \times \Teich^\hyp(\Sigma)$, then we denote by
\[
\varphi = \varphi^{+} + \varphi^{-} \in T^*_p \Teich^\hyp(\Sigma) \oplus T^*_q \Teich^\hyp(\Sigma) 
\]
its decomposition through the isomorphism $T_{(p,q)}^* \Teich^\hyp(\Sigma) \times \Teich^\hyp(\Sigma) \cong T^*_p \Teich^\hyp(\Sigma) \oplus T^*_q \Teich^\hyp(\Sigma)$.
For every $k$ we have
\begin{align*}
	\abs{V_k^*(M) - (V_k^* \circ \Psi_k^{-1} \circ \gamma_k)(0)} & = \abs{(V_k^* \circ \Psi_k^{-1} \circ \gamma_k)(1) - (V_k^* \circ \Psi_k^{-1} \circ \gamma_k)(0)} \\
	& = \abs{ \int_0^1 \dv{t} (V_k^* \circ \Psi_k^{-1} \circ \gamma_k)(t) \dd{t} } \\
	& \leq \int_0^1 \abs{\dd{(V_k^* \circ \Psi_k^{-1})}^+_{\gamma_k(t)} (\beta_k'(t)) + 0 } \dd{t} \\
& \leq \max_{t} \norm{\dd{(V_k^* \circ \Psi_k^{-1})}^+_{\gamma_k(t)}}_\WP \ \length_\WP(\beta_k) \\
& = \max_{t} \norm{\dd{(V_k^* \circ \Psi_k^{-1})}^+_{\gamma_k(t)}}_\WP \ d_\WP(m_k^+, m_k^-),
\end{align*}
where $\norm{\cdot}_\WP$ denotes the Weil-Petersson norm on $T^* \Teich^\hyp(\Sigma)$. The step from the second to the third line follows from the fact that the second component of the curve $\gamma_k$ does not depend on $t$, and in the last step we used that $\beta_k$ is a Weil-Petersson geodesic.

By Theorem \ref{thm:k_surfaces_foliation}, the $k$-\hsk surfaces of $M$ approach the boundary of the convex core $\partial \CC M$ as $k$ goes to $-1$. In particular, the isotopy classes of their (normalized) fundamental forms $(- k) \I_k$ converge to the hyperbolic structures on the boundary of the convex core (see \cite[Section 3.6]{labourie1992surfaces} or the proof of Lemma \ref{lem:compact} below for an explanation of this fact). Therefore we have
\[
\lim_{k \to -1} d_\WP(m_k^+,m_k^-) = d_\WP(m^+,m^-) ,
\]
where $m^+$, $m^-$ are the hyperbolic metrics of the upper and lower components of $\partial \CC M$, respectively. Similarly, being the dual volume continuous with respect to the Hausdorff topology on compact convex subsets of $M$ (see \cite[Proposition~2.3]{mazzoli2021the_dual}), we have
\[
\lim_{k \to -1} V_k^*(M) = V_\CC^*(M) .
\]
If $M'$ is a Fuchsian manifold homeomorphic to $\Sigma \times \R$, the $k$-\hsk surfaces of $M'$ coincide with the $\varepsilon_k$-\hsk equidistant surfaces from the convex core, where $\varepsilon_k = \tanh^{-1}(\sqrt{k + 1})$, and the dual volume $V_k^*(M')$ can be expressed as
\[
V_k^*(M') = - \pi \abs{\chi(\Sigma)} (\sinh 2\varepsilon_k - 2 \varepsilon_k ) 
\]
(see e. g. \cite[Proposition 2.4]{mazzoli2021the_dual}). In particular the term $V_k^*(M')$ is infinitesimal as $k$ goes to $-1$, uniformly in $M'$. These observations combined with the inequalities above allows us to deduce
\begin{equation} \label{eq:iintermediate_thm}
	\abs{V_\CC^*(M)} \leq \liminf_{k \to -1} \max_t \norm{\dd{(V_k^* \circ \Psi_k^{-1})}^+_{\gamma_k(t)}}_\WP \ d_\WP(m^+,m^-) .
\end{equation}
In order to simplify the notation, we now introduce maps
\[
\sigma_k \vcentcolon \QF(\Sigma) \longrightarrow T^* \Teich^\hyp(\Sigma) \times \Teich^\hyp(\Sigma) , \qquad \sigma \vcentcolon \QF(\Sigma) \longrightarrow T^* \Teich^\hyp(\Sigma) \times \Teich^\hyp(\Sigma) ,
\]
defined as
\begin{equation} \label{eq:def_sigmak}
	\sigma_k^\pm(M) \defin \dd{(V_k^* \circ \Psi_k^{-1})}^{\pm}_{\Psi_k(M)} \in T^*_{m_k^\pm(M)} \Teich^\hyp(\Sigma),
\end{equation}
and
\begin{equation} \label{eq:def_sigma}
	\sigma^\pm(M) \defin - \frac{1}{2} \dd{(L_{\mu^\pm(M)})}_{m^\pm(M)} \in T^*_{m^\pm(M)} \Teich^\hyp(\Sigma) ,
\end{equation}
where $m^\pm(M)$ and $\mu^\pm(M)$ denote the hyperbolic metric and the bending measured lamination on the upper/lower components of the boundary of the convex core of $M$, respectively, and $L_{\mu^\pm(M)}$ stands for the hyperbolic length function of $\mu^\pm(M)$ over the Teichm\"uller space of $\Sigma$. We now claim that Theorem \ref{thm:bound_dual_volume_WP} is a direct consequence of relation \eqref{eq:iintermediate_thm}, the results from the previous sections and the following facts:

\begin{lemma} \label{lem:compact}
	For any $k_0 \in (-1,0)$, the family of paths $\Psi_k^{-1} \circ \gamma_k$, as $k$ varies in $(-1,k_0]$, lies in a common compact subset of the space of quasi-\hsk Fuchsian manifolds $\QF(\Sigma)$.
\end{lemma}

\begin{proposition} \label{prop:uniform_convergence_diff}
	The maps $\sigma_k$, as functions on the space of quasi-\hsk Fuchsian structures $\QF(\Sigma)$ on $\Sigma \times \R$ with values in the cotangent space to Teichm\"uller space $T^* \Teich^\hyp(\Sigma) \times \Teich^\hyp(\Sigma)$, converge to $\sigma$ uniformly over all compact sets of $\QF(\Sigma)$ as $k$ goes to $-1$.
\end{proposition}

Assuming temporarily these facts, we describe how to obtain the desired statement. By Lemma \ref{lem:compact} we can find a compact subset of $\QF(\Sigma)$ containing the curves $\Psi_k^{-1} \circ \gamma_k$, for $k$ smaller than some fixed $k_0 \in (-1,0)$. Then, by Proposition \ref{prop:uniform_convergence_diff} we have
\[
\liminf_{k \to -1} \max_t \norm{\dd{(V_k^* \circ \Psi_k^{-1})}^+_{\gamma_k(t)}}_\WP = \liminf_{k \to -1} \max_t \norm{\sigma_k^+(\Psi_k^{-1} \circ \gamma_k(t))}_\WP \leq \sup_{\QF(\Sigma)} \norm{\sigma^+}_\WP .
\]
To control the supremum of $\norm{\sigma^+}_\WP$ over $\QF(\Sigma)$, we proceed similarly to what done in the proof of Corollary \ref{cor:bound_diff_dual_vol}. 
Consider $M'$ a quasi-\hsk Fuchsian manifold homeomorphic to $\Sigma \times \R$, with upper bending measure ${\mu'}^+ = \mu^+(M')$ and hyperbolic metric on the upper component of $\partial \CC M'$ equal to ${m'}^+ = m^+(M')$. First we apply Proposition \ref{prop:bound_differential_length} to bound the Weil-Petersson norm of the differential of the length
\[
\norm{\sigma^+(M')}_\WP = \frac{1}{2} \norm{\dd{(L_{{\mu'}^+})}_{{m'}^+}}_{Q,2} \leq C(2,r) \ L_{{\mu'}^+}({m'}^+)^{1/2} \ D({\mu'}^+, {m'}^+, r)^{1/2} .
\]
Then we estimate the term $L_{{\mu'}^+}({m'}^+)^{1/2}$ using Theorem \ref{thm:bound_length}, and the coefficient $D({m'}^+, {\mu'}^+, r)^{1/2}$ via Corollary \ref{cor:bound_total_length}. Finally we take a limit for $r$ that goes to $\ln(3)/2$, obtaining
\[
	\norm{\sigma^+({M'})}_\WP \leq \frac{K(2)}{\sqrt{2}}(g - 1)^{1/2}
\]
for every quasi-Fuchsian manifold $M' \in \QF(\Sigma)$, where $K(2)$ is the constant appearing in Corollary \ref{cor:bound_diff_dual_vol} (the factor $1/\sqrt{2}$ is due to the fact that we applied Theorem \ref{thm:bound_length} only to the upper component of the boundary of the convex core). This remark, combined with relation \eqref{eq:iintermediate_thm}, finally shows that
\[
\abs{V_\CC^*(M)} \leq C (g - 1)^{1/2} \ d_\WP(m^+,m^-) ,
\]
where $C = K(2) / \sqrt{2} \approx 7.3459$, as stated in Theorem \ref{thm:bound_dual_volume_WP}.

\medskip

We can now focus on the proof of Lemma \ref{lem:compact} and Proposition \ref{prop:uniform_convergence_diff} that were applied above. We briefly remind the definition of the Thurston's asymmetric distance on Teichm\"uller space and a compactness criterion that will be useful in our argument. Given $h$ and $h'$ two hyperbolic metrics on $\Sigma$ and given $\mappa{\varphi}{\Sigma}{\Sigma}$ a diffeomorphism isotopic to the identity, we set $\textrm{Lip}(\varphi)$ to denote the Lipschitz constant of $\varphi$, i. e.
\[
\textrm{Lip}(\varphi) \defin \sup_{x \neq y} \frac{d_{h'}(\varphi(x), \varphi(y))}{d_h(x,y)} .
\]
Then we define $d_{\textrm{Th}}(h,h')$ to be
\[
d_{\textrm{Th}}(h,h') \defin \log \ \inf_{\varphi} \ \textrm{Lip}(\varphi) ,
\]
where the infimum is taken over all diffeomorphisms $\varphi$ isotopic to the identity. The quantity $d_{\textrm{Th}}(h,h')$ only depends on the isotopy classes of the metrics $h$ and $h'$, so the above definition provides a function on $\Teich^\hyp(\Sigma) \times \Teich^\hyp(\Sigma)$, which we continue to denote by $d_{\textrm{Th}}$. As shown by Thurston in \cite{thurston1998minimal}, $d_{\textrm{Th}}$ satisfies all the properties of a distance except for the symmetry: there are choices of hyperbolic structures $m, m' \in \Teich^\hyp(\Sigma)$ for which $d_{\textrm{Th}}(m,m')$ is different from $d_{\textrm{Th}}(m',m)$. The distance $d_{\textrm{Th}}(m,m')$ can be equivalently characterized in terms of the length spectra of the hyperbolic structures $m$ and $m'$. More precisely, we have that
\[
d_{\textrm{Th}}(m,m') = \sup_\gamma \frac{\length_{m'}(\gamma)}{\length_{m}(\gamma)} ,
\] 
where $\gamma$ varies among the isotopy classes of simple closed curves of $\Sigma $, and $\length_{m}(\gamma)$, $\length_{m'}(\gamma)$ denote the lengths of the $m$- and $m'$-\hsk geodesic representatives of $\gamma$. The Thurston distance $d_{\textrm{Th}}(m,m')$ can be considered as a Riemannian analogue of the notion of Teichm\"uller distance $d_\Teich(c,c')$, where the hyperbolic structures $m$, $m'$ are taking the roles of the conformal structures $c$, $c'$, and Lipschitz constant $\textrm{Lip}(\varphi)$ replaces the quasi-\hsk conformal dilatation of $\varphi$. A similar characterization in terms of the length spectra holds also for Teichm\"uller distance, where the hyperbolic length functions $\ell_m$, $\ell_{m'}$ are replaced by the extremal length functions with respect to $c$ and $c'$, respectively. We will make use of the following compactness criterion for the Thurston's distance, shown by \citet{papadopoulos2007thurston}:

\begin{theorem}[{\cite[Theorem~2]{papadopoulos2007thurston}}] \label{thm:compactness_thurston}
	For every sequence $(m_n)_n$ in $\Teich^\hyp(\Sigma)$, the following are equivalent:
	\begin{enumerate}
		\item The sequence $(m_n)_n$ leaves every compact subset in $\Teich^\hyp(\Sigma)$;
		\item for every $m' \in \Teich^\hyp(\Sigma)$, the distance $d_{\textrm{Th}}(m',m_n)$ goes to $+ \infty$;
		\item for every $m' \in \Teich^\hyp(\Sigma)$, the distance $d_{\textrm{Th}}(m_n,m')$ goes to $+ \infty$.
	\end{enumerate}
\end{theorem}
	We now have all the elements for the proof of Lemma \ref{lem:compact}.

\begin{proof}[Proof of Lemma \ref{lem:compact}]
	Following the notation introduced in Section \ref{section:dual_volume}, we denote by
	\[
	B \vcentcolon \QF(\Sigma) \longrightarrow \Teich^\conf(\Sigma) \times \Teich^\conf(\overline{\Sigma})
	\]
	the Bers' homeomorphism and by 
	\[
	\Psi \vcentcolon \QF(\Sigma) \longrightarrow \Teich^\hyp(\Sigma) \times \Teich^\hyp(\Sigma)
	\]
	the function that maps a quasi-\hsk Fuchsian manifold into the pair of hyperbolic structures on the boundary of its convex core. By the work of Sullivan \cite{sullivan1981travaux}, and Epstein-Marden \cite{epstein2006fundamentals}, we can find a universal constant $K > 0$ such that, for every quasi-Fuchsian manifold $M' \in \QF(\Sigma)$, there exists a $K$-\hsk quasiconformal homeomorphism isotopic to the identity between the conformal structures at infinity and the hyperbolic structures on the boundary of the convex core of $M'$ (see also \cite{yarmola2016improved} for improved bounds on the constant $K$). In particular, if $d_{\Teich}$ denotes the Teichm\"uller distance on $\Teich(\Sigma \sqcup \Sigma) =  \Teich(\Sigma) \times \Teich(\Sigma)$, we have that
	\[
	\sup_{M' \in \QF(\Sigma)} d_{\Teich}(B(M'), \Psi(M')) < \infty .
	\]
	Because the Teichm\"uller metric is complete and $B$ is a homeomorphism, the paths $(\Psi_k^{-1} \circ \gamma_k)_{k \leq k_0}$ lie in a common compact subset of $\QF(\Sigma)$ if and only if $(\Psi \circ \Psi_k^{-1} \circ \gamma_k)_{k \leq k_0}$ lie in a common compact subset of $\Teich^\hyp(\Sigma) \times \Teich^\hyp(\Sigma)$. In the remainder of the proof we will show this last property.
	
	Given $M'$ a quasi-\hsk Fuchsian manifold, we denote by $M_k'$ the convex subset of $M'$ bounded by its $k$-\hsk surfaces. The metric retraction on the convex hull of the limit set in the universal cover of $M'$ induces a surjective $1$-\hsk Lipschitz map $R_k$ from the $k$-\hsk surfaces $\partial M_k'$ to the boundary of its convex core. Being $m_k^\pm(M')$ the isotopy classes of the hyperbolic metrics $(- k) \I_k$, with $\I_k$ the first fundamental form of $\partial M_k'$, we deduce that the function
	\[
	R_k \vcentcolon (\partial M_k', m_k^\pm(M')) \longrightarrow (\partial \CC M', m^\pm(M'))
	\]
	is $(-k^{-1})$-\hsk Lipschitz. Similarly, by considering the metric retraction $r_k$ onto the convex subset $M_k'$, we have that the map
	\[
	r_k \vcentcolon (\partial M_{k_0}', m_{k_0}^\pm(M')) \longrightarrow (\partial M_k', m_k^\pm(M'))
	\]
	is $(k k_0^{-1})$-\hsk Lipschitz. Observe that both $R_k$ and $r_k$ respect the markings of the hyperbolic structures $m^\pm(M')$, $m_k^\pm(M')$ and $m_{k_0}^\pm(M')$. This tells us in particular that, if $d_{\textrm{Th}}$ denotes the Thurston's asymmetric distance on Teichm\"uller space, then
	\begin{align}
		d_{\textrm{Th}}(m^\pm(M'), m_k^\pm(M')) & \leq \log (-k^{-1}) , \label{eq:bound1} \\
		d_{\textrm{Th}}(m_k^\pm(M'), m_{k_0}^\pm(M')) & \leq \log (k k_0^{-1}) , \label{eq:bound2}
	\end{align}
	for every quasi-\hsk Fuchsian manifold $M' \in \QF(\Sigma)$. 
	Remember now that the curves $\gamma_k$ were defined so that $\gamma_k(t) = (\beta_k(t), m_k^-(M))$, where $\beta_k$ is the Weil-Petersson geodesic connecting $m_k^+(M)$ to $m_k^-(M)$, for some \emph{fixed} quasi-\hsk Fuchsian manifold $M$. By relation \eqref{eq:bound2} and Theorem \ref{thm:compactness_thurston}, there exists a compact subset $C$ of $\Teich^\hyp(\Sigma)$ containing $\set{m_k^+(M)}_{k \leq k_0} \cup \set{m_k^-(M)}_{k \leq k_0}$. Observe also that relation \eqref{eq:bound1} implies that $\lim_{k \to -1} m_k^\pm(M') = m^\pm(M')$ for every $M'$.
	
	By a result of Wolpert \cite[Corollary~5.6]{wolpert1987geodesic}, the Teichm\"uller space admits an exaustion by compact Weil-\hsk Petersson convex sets. Being the endpoints of the geodesic $\beta_k$ contained inside the compact $C$, we can find a possibly larger compact Weil-\hsk Petersson convex region $C'$ containing the images of the paths $(\beta_k)_{k \leq k_0}$. The curves $(\gamma_k)_{k \leq k_0}$ will then be lying inside $C' \times C \subset \Teich^\hyp(\Sigma) \times \Teich^\hyp(\Sigma)$. In order to conclude, we observe that relation \eqref{eq:bound1} can be restated as
	\[
	d_{\textrm{Th}}(\Psi \circ \Psi_k^{-1} (X), X) \leq \log(-k^{-1})
	\]
	for every $X \in \Teich^\hyp(\Sigma \sqcup \Sigma) \cong \Teich^\hyp(\Sigma) \times \Teich^\hyp(\Sigma)$. Therefore, applying once again Theorem \ref{thm:compactness_thurston} and knowing that the curves $(\gamma_k)_{k \leq k_0}$ are contained in a compact set of $\Teich^\hyp(\Sigma \sqcup \Sigma)$, we deduce the existence of a compact region containing the paths $\Psi \circ \Psi_k^{-1} \circ \gamma_k$ for all $k \leq k_0$, as desired.
\end{proof}

The last step left is to prove Proposition \ref{prop:uniform_convergence_diff}. We will deduce this fact from the so called \emph{dual differential Schl\"afli formula}, stated in Theorem \ref{thm:dual_diff_schlafli}, and from the connection between the first order variation of the volume functions $V_k^* \circ \Psi_k^{-1}$ and the notion of landslides introduced in \cite{bonsante2013a_cyclic}, \cite{bonsante2015a_cyclic}.

\begin{theorem}[{\cite{schlenker_rivin1999schlafli}}]\label{thm:dual_diff_schlafli}
Let $N$ be a compact manifold with boundary, and assume that there exists a smooth $1$-parameter family $(g_t)_t$ of hyperbolic metrics with strictly convex boundary on $N$. Then there exists the derivative of $t \mapsto \Vol^*(N,g_t)$ and it satisfies
\[
\dv{t} \left. \Vol^*(N,g_t) \right|_{t = 0} = \frac{1}{4} \int_{\partial N} \scall{\delta g|_{\partial N}}{H \I - \II} \dd{a}_\I ,
\]
where $\delta g = \dv{t} g_t |_{t = 0}$.
\begin{proof}
This relation is a corollary of \cite[Theorem~8]{schlenker_rivin1999schlafli}. It is enough to apply this result to the definition of dual volume $\Vol^*(N,g_t)$, together with the relation
\[
\delta \left( \int_{\partial N} H \dd{a} \right) = \int_{\partial N} \left( \delta H + \frac{H}{2} \scall{\delta \I}{\I} \right) \dd{a}_\I ,
\]
which follows by differentiating the expression $H \dd{a} = H \sqrt{\det \I} \dd{x} \wedge \dd{y}$ in local coordinates.
\end{proof}
\end{theorem}

\subsection{Earthquakes and landslides}

We briefly recall the definition of landslide flow, introduced in \citet{bonsante2013a_cyclic}, and the properties that we will need for the proof of Proposition \ref{prop:uniform_convergence_diff}.

\vspace{0.5cm}

Landslides are described by a map
\[
\begin{matrix}
\altmathcal{L} \vcentcolon & S^1 \times \Teich^\hyp(\Sigma) \times \Teich^\hyp(\Sigma) & \longrightarrow & \Teich^\hyp(\Sigma) \times \Teich^\hyp(\Sigma)  \\
& (e^{i \theta}, m, m') & \longmapsto & \altmathcal{L}_{e^{i \theta}}(m,m') .
\end{matrix}
\]
The first component of $\altmathcal{L}_{e^{i \theta}}(m,m')$, which we will denote by $\altmathcal{L}^1_{e^{i \theta}}(m,m')$, is called the \emph{landslide of $m$ with respect to $m'$ with parameter $e^{i \theta}$}. The map $\altmathcal{L}$ is defined via the following result:
\begin{theorem}[{\cite{labourie1992surfaces},\cite{shoen1993the_role}}] \label{thm:minimal_lagrangian}
	Let $m, m' \in \Teich^\hyp(\Sigma)$. Then, for any representative $h \in m$, there exists a unique hyperbolic metric $h' \in m'$ and a unique endomorphism $\mappa{b}{T \Sigma}{T \Sigma}$ such that:
	\begin{itemize}
		\item $h'(\cdot,\cdot) = h(b \cdot, b \cdot)$;
		\item $b$ is $h$-\hsk self-\hsk adjoint and positive definite;
		\item $b$ has determinant $1$;
		\item $b$ is Codazzi with respect to the Levi-\hsk Civita connection $\nabla$ of $h$, i. e. $(\nabla_X b)Y = (\nabla_Y b)X$ for all $X$, $Y$.
	\end{itemize}
\end{theorem}
\noindent The operator $b$ is also called the \emph{Labourie operator} of the couple $h$, $h'$. In the following, we will identify, with abuse, a pair of isotopy classes $m$, $m' \in \Teich^\hyp(\Sigma)$ with a pair of hyperbolic metrics $h$, $h'$ satisfying the conclusions of the Theorem above. Given $\theta \in \R / 2\pi \Z$ and two hyperbolic metrics $h$, $h'$ with Labourie operator $b$, we denote by $b^\theta$ the endomorphism $\cos(\theta/2) \1 + \sin(\theta/2) J b$, where $J$ is the almost complex structure of $h$, and we set $h^\theta \defin h(b^\theta \cdot,b^\theta \cdot)$. Then the function $\altmathcal{L}$ is defined as:
\[
\altmathcal{L}_{e^{i \theta}}(h,h') \defin (h^\theta, h^{\pi + \theta}) .
\]
It turns out that, for any $\theta$, the metric $h^\theta$ is hyperbolic, and $\altmathcal{L}$ actually defines a flow, in the sense that it satisfies $\altmathcal{L}_{e^{i \theta}} \circ \altmathcal{L}_{e^{i \theta'}} = \altmathcal{L}_{e^{i (\theta + \theta')}}$ for all $\theta, \theta'$.

As earthquakes extend to \emph{complex earthquakes} (see \cite{mcmullen1998complex}), similarly happens for landslides. Fixed $h, h' \in \Teich^\hyp(\Sigma)$, the map $\altmathcal{L}^1_\bullet(h,h')$ extends to a holomorphic function $C_\bullet(h,h')$ defined on a open neighborhood of the closure of the unit disc $\Delta$ in $\C$. If $\zeta = \exp(s + i \theta) \in \overline{\Delta}$, then $C_\zeta$ can be written as
\[
C_\zeta(h,h') = \sgr_s \circ \altmathcal{L}_{e^{i \theta}} (h,h') ,
\]
where $\mappa{\sgr_s}{\Teich^\hyp(\Sigma) \times \Teich^\hyp(\Sigma)}{\Teich^\hyp(\Sigma)}$ is called the \emph{smooth grafting map}. If $s = 0$, then $\sgr_0 \circ \altmathcal{L}_{e^{i \theta}} = \altmathcal{L}^1_{e^{i \theta}}$.

Constant Gaussian curvature surfaces are a natural example in which pairs of metrics as in Theorem \ref{thm:minimal_lagrangian} arise. Let $\Sigma_k$ be a $k$-\hsk surface in a hyperbolic $3$-\hsk manifold, with first fundamental form $\I_k$ and shape operator $B_k$. The \emph{third fundamental form} of $\Sigma_k$ is defined as $\III_k \defin \I_k(B_k \cdot, B_k \cdot)$. Either by direct computation or by using the duality correspondence between hypersurfaces of $\Hyp^3$ and $\dS^3$ (see \cite{rivin1986phd}, \cite{schlenker2002hypersurfaces}), we see that the third fundamental form is a constant Gaussian curvature metric too, with curvature $\frac{k}{k + 1}$. Moreover, if we set
\begin{equation} \label{eq:hyp_metrics_k_surfaces}
h_k \defin - k \ \I_k, \qquad h_k' \defin - \frac{k}{k + 1} \ \III_k, \qquad b_k \defin \frac{1}{\sqrt{k + 1}} B_k ,
\end{equation}
then $h_k$ and $h_k' = h_k(b_k \cdot, b_k \cdot)$ are hyperbolic metrics satisfying the properties of Theorem \ref{thm:minimal_lagrangian}. We refer to \cite{bonsante2013a_cyclic} and \cite{bonsante2015a_cyclic} for a more detailed exposition about landslides, and to \citet{labourie1992surfaces} for what concerns $k$-\hsk surfaces.

Fixed $h'$, we set $l^1(h,h')$ to be the infinitesimal generator of the landslide flow with respect to the hyperbolic metric $h'$ at the point $h \in \Teich^\hyp(\Sigma)$. In other words,
\[
l^1(h,h') \defin \dv{\theta} \left. \altmathcal{L}^1_{e^{i \theta}}(h,h') \right|_{\theta = 0} \in T_h \Teich^\hyp(\Sigma) .
\]
Landslides extend the notion of earthquake in the sense explained by the following Theorem:

\begin{theorem}[{\cite[Proposition~6.9]{bonsante2013a_cyclic}}] \label{thm:limit_landslides}
	Let $(h_n)_n$ and $(h_n')_n$ be two sequences of hyperbolic metrics on $\Sigma$ such that $(h_n)_n$ converges to $h \in \Teich^\hyp(\Sigma)$, and $(h_n')_n$ converges to a projective class of measured lamination $[\mu]$ in the Thurston boundary of Teichm\"uller space. 
	If $(\theta_n)_n$ is a sequence of positive numbers such that the length spectra $\theta_n \ \length_{h_n'}$ converge to transverse measure $\mu$, then $\altmathcal{L}^1_{e^{i \theta_n}}(h_n,h_n')$ converges to the left earthquake $\altmathcal{E}_{\mu/2}(h)$, and ${\theta_n \ l^1(h_n,h_n')|}_{h_n}$ converges to the infinitesimal earthquake $\frac{1}{2} {e_{\mu}|}_h = \dv{t} \altmathcal{E}_{t \mu/2}(h) |_{t = 0}$.
\end{theorem}

\begin{remark}
The last part of the assertion follows from the fact that the functions $e^{i \theta} \mapsto \altmathcal{L}^1_{e^{i \theta}}(h,h')$ extend to holomorphic functions $\zeta \mapsto C_\zeta(h,h')$, where $\zeta$ varies in a neighborhood of $\overline{\Delta}$. In particular, the uniform convergence of the complex landslides $C_\bullet(h_n,h_n')$ to the complex earthquake map implies uniform convergence in the $\mathscr{C}^\infty$-\hsk topology with respect to the complex parameter $\zeta$.
\end{remark}

In order to prove the relation between the differential of the functions $V_k^* \circ \Psi_k^{-1}$ (and therefore the maps $\mappa{\sigma_k}{\QF(\Sigma)}{T^* \Teich^\hyp(\Sigma) \times \Teich^\hyp(\Sigma)}$ defined in \eqref{eq:def_sigmak}) and the landslide flow, it will be useful to have an explicit expression to compute the variation of the hyperbolic length of a simple closed curve $\alpha$ of $\Sigma$ along the infinitesimal landslide $l^1(h,h')$.

\begin{lemma} \label{lem:length_infinitesimal_landslide}
Let $\alpha$ be a simple closed curve in $\Sigma$. Then we have
\[
\dv{\theta} \left. L_\alpha(\altmathcal{L}^1_{e^{i \theta}}(h,h')) \right|_{\theta = 0} = - \int_\alpha \frac{h(b \alpha',J \alpha')}{2 \norm{\alpha'}^2_h} \dd{\ell_h} ,
\]
where $J$ is the complex structure of $h$ and $b$ is the Labourie operator of the couple $h$, $h'$.
\begin{proof}
With abuse, we denote the $h$-\hsk geodesic realization of $\alpha$ by $\alpha$ itself. By definition of landslide we have 
\[
\dv{\theta} \left. \altmathcal{L}^1_{e^{i \theta}}(h,h')(\alpha',\alpha') \right|_{t = 0} = \dot{h}(\alpha',\alpha') = h(\alpha',J b \alpha') .
\]
Since $J$ is $h$-\hsk orthogonal and $J^2 = - \id$, we deduce that $\dot{h}(\alpha',\alpha') = - h(b \alpha',J \alpha')$. Combining this relation with Lemma \ref{lem:derivative_length_integral} we obtain the statement.
\end{proof}
\end{lemma}

\noindent In order to simplify the notation, we will write $\I_k$, $\II_k$ and $\III_k$ to denote the fundamental forms of the $k$-\hsk surface $\partial M_k = \Sigma_k^+ \sqcup \Sigma_k^-$ (these are tensors on the \emph{union} of the upper and lower $k$-\hsk surfaces). In particular, $h_k$ and $h_k'$ will represent the hyperbolic metrics on $\Sigma^+ \sqcup \Sigma^-$ defined as in \eqref{eq:hyp_metrics_k_surfaces}. The relation between landslides and the dual volume of the region enclosed by the two $k$-\hsk surfaces is described by the following fact:

\begin{proposition} \label{prop:gradient_dual_volume_k}
Let $M$ be a quasi-Fuchsian manifold and let $h_k$, $h_k' \in \Teich^\hyp(\Sigma^+ \sqcup \Sigma^-)$ denote the hyperbolic metrics $- k \ \I_k$ and $- k (k + 1)^{-1} \ \III_k$. Then we have
\[
\sigma_k(M) = \sqrt{- \frac{k + 1}{k}} \ \omega_\WP(l^1(h_k,h_k'),\cdot) \in T^*_{\Psi_k(M)} \Teich^\hyp(\Sigma^+ \sqcup \Sigma^-) ,
\]
where $\omega_\WP$ is the Weil-\hsk Petersson symplectic form on $\Teich^\hyp(\Sigma^+ \sqcup \Sigma^-) \cong \Teich^\hyp(\Sigma) \times \Teich^\hyp(\Sigma)$ and $\sigma_k$ is the map defined in relation \eqref{eq:def_sigmak}.
\begin{proof}
Given a simple closed curve $\alpha$ in $\Sigma^+ \sqcup \Sigma^-$, we denote by $e_\alpha$ the infinitesimal generator of the left earthquake flow along $\alpha$ on $\Teich^\hyp(\Sigma^+ \sqcup \Sigma^-)$. We will prove the statement by showing that, for every simple closed curve $\alpha$, we have
\begin{equation} \label{eq:diff_dual_k_volume_1}
\sigma_k(M) (e_\alpha) = \dd{(V_k^* \circ \Psi_k^{-1})}_{\Psi_k(M)}(e_\alpha) = \sqrt{- \frac{k + 1}{k}} \ \omega_\WP(l^1(h_k,h_k'),e_\alpha) ,
\end{equation}
where $\sigma_k(M)(e_\alpha)$ stands for evaluation of the cotangent vector $\sigma_k(M) \in T^*_{\Psi_k(M)} \Teich^\hyp(\Sigma^+ \sqcup \Sigma^-)$ at $e_\alpha \in T_{\Psi_k(M)} \Teich^\hyp(\Sigma^+ \sqcup \Sigma^-)$.
Since the constant $k$ will be fixed, we will not write the dependence on $k$ in the objects involved in the argument, in order to simplify the notation. By Theorem \ref{thm:ksurfaces_homeo}, given any first order variation of metrics $\delta \I = \delta \I_k$ on the $k$-\hsk surface $\Sigma^+ \sqcup \Sigma^- = \Sigma^+_k \sqcup \Sigma^-_k$, we can find a variation $\delta g$ of hyperbolic metrics on $M$ satisfying $\dd{\Psi_k}(\delta g) = \delta g |_{\Sigma^+ \sqcup \Sigma^-} = \delta \I$. Our first step will be to construct an explicit variation $\delta \I$ corresponding to the vector field $e_\alpha$, and then to apply Proposition \ref{thm:dual_diff_schlafli} to compute $\dd{(V^*_k \circ \Psi_k^{-1})}(e_\alpha)$.
 
We will identify the curve $\alpha$ with its $\I$-\hsk geodesic parametrization of length $L_\alpha$ and at speed $1$. Let $J$ denote the almost complex structure of $\I$, and set $V$ to be the vector field along $\alpha$ given by $- J \alpha'$ ($J \alpha'$ is a vector field along $\alpha$ that is orthogonal to $\alpha'$). We can find a $\varepsilon > 0$ so that the map
\[
\begin{matrix}
\xi \vcentcolon & \R / L_\alpha \Z \times [0,\varepsilon] & \longrightarrow & \Sigma \\
& (s,r) & \longmapsto & \exp_{\alpha(s)}(r V(s))
\end{matrix}
\]
is a diffeomorphism onto its image (here $\exp$ is the exponential map with respect to $\I$). The image of $\xi$ is a closed cylinder in $\Sigma$ having $\alpha$ as left boundary component. Observe that the metric $\I$ equals $\dd{r}^2 + \cosh^2 r \dd{s}^2$ in the coordinates defined by $\xi^{-1}$. We also choose a smooth function $\mappa{\eta}{[0,\varepsilon]}{[0,1]}$ that coincides with $1$ in a neighborhood of $0$, and with $0$ in a neighborhood of $\varepsilon$. Now define
\[
\begin{matrix}
f_t \vcentcolon & \R/L_\alpha \Z \times [0,\varepsilon] & \longrightarrow & \R/L_\alpha \Z \times [0,\varepsilon] \\
& (s,r) & \longmapsto & (s + t \eta(r),r) .
\end{matrix}
\]
The maps $u_t \defin \xi \circ f_t \circ \xi^{-1}$ give a smooth isotopy of the strip $\Im \xi$ adjacent to $\alpha$, with $u_0 = \id$. Finally we set
\[
\delta \I \defin
\begin{cases}
\left. \dv{t} u_t^* \I \right|_{t = 0} = 2 \eta'(r) \cosh^2 r \ \dd{r} \dd{s} & \text{inside $\Im \xi$}, \\
0 & \text{elsewhere},
\end{cases}
\]
where here $2 \dd{s} \dd{r} =  \dd{s} \otimes \dd{r} + \dd{r} \otimes \dd{s}$. Thanks to our choice of the function $\eta$, $\delta \I$ is a smooth symmetric tensor of $\Sigma^+ \sqcup \Sigma^-$ that represents the first order variation of $\I$ along the infinitesimal left earthquake $e_\alpha$. By Proposition \ref{thm:dual_diff_schlafli}, we have that
\[
\dd{V^*_k}(\delta g) = \frac{1}{4} \int_{\Sigma_k^+ \sqcup \Sigma_k^-} \scall{\delta g|_{\Sigma_k^+ \sqcup \Sigma_k^-}}{H \I - \II} \dd{a} = - \frac{1}{4} \int_0^{L_\alpha} \int_0^\varepsilon \scall{\delta \I}{\II} \cosh r \dd{r} \dd{s} ,
\]
where the last step follows from the fact that $\delta \I$ is $\I$-\hsk traceless. Let $\nabla$ denote the Levi-Civita connection of $\I$. Then the coordinate vector fields of $\xi^{-1}$ satisfy:
\[
\nabla_{\partial_r} \partial_r = 0, \qquad \nabla_{\partial_s} \partial_r = \nabla_{\partial_r} \partial_s = \tanh r \ \partial_s , \qquad \nabla_{\partial_s} \partial_s = - \sinh r \cosh r \ \partial_r .
\]
By definition, $\scall{\delta \I}{\II} = 2 \ \I^{rr} \, \I^{ss} \, \delta \I_{rs} \, \II_{r s} = 2 \eta' \ \II_{rs}$. If we set $f(r) \defin \int^{L_\alpha}_0 \II_{r s} \dd{s}$, then, integrating by parts and recalling that $\eta(\varepsilon) = 0$, we get
\begin{align*}
\dd{V^*_k}(\delta g) & = - \frac{1}{2} \int_0^\varepsilon \eta'(r) f(r) \cosh r \dd{r} \\
& = \frac{1}{2} f(0) + \frac{1}{2} \int_0^\varepsilon \eta(r) (f'(r) \cosh r + f(r) \sinh r) \dd{r} \tag{$\star$}
\end{align*}
Being the second fundamental form a Codazzi tensor, we have $(\nabla_{\partial_r} \II)_{r s} = (\nabla_{\partial_s} \II)_{r r}$. Using the expressions of the connection given above, this relation can be rephrased as $\partial_r \II_{r s} = \partial_s \II_{r r} - \tanh r \ \II_{s r}$. Hence we deduce
\[
f'(r) = \int^{L_\alpha}_0 ( \partial_s \II_{r r} - \tanh r \ \II_{s r} ) \dd{s} = - \tanh r \ f(r) ,
\]
where the first summand vanishes because $\alpha$ is a closed curve. Therefore the integral in the relation ($\star$) equals $0$, and we end up with the equation
\begin{equation} \label{eq:diff_dual_k_volume_2}
\dd{V^*_k}(\delta g) = \frac{1}{2} \int_0^{L_\alpha} \II_{rs} \dd{s} =  -\frac{1}{2} \int_0^{L_\alpha} \I(B \alpha', J \alpha') \dd{s}
\end{equation}
since $\partial_r |_{r = 0} = V = - J \alpha'$ and $\partial_s |_{r = 0} = \alpha'$.

Now we apply Lemma \ref{lem:length_infinitesimal_landslide} to $\alpha$, the hyperbolic metrics $h = - k \ \I$, $h' = - \frac{k}{k + 1} \ \III$ and the operator $b = \frac{1}{\sqrt{k + 1}} B$ (here $B$ is the shape operator of $\Sigma_k^+ \sqcup \Sigma_k^-$), obtaining 
\[
\dd{(L_\alpha)}_h (l^1(h,h')) = - \frac{1}{2} \sqrt{- \frac{k}{k + 1}} \int_0^{L_\alpha} \I(B \alpha', J \alpha') \dd{s} .
\]
This relation, combined with \eqref{eq:diff_dual_k_volume_2}, proves that
\[
\dd{V^*_k}(\delta g) = \sqrt{- \frac{k + 1}{k}} \ \dd{(L_\alpha)}_h (l^1(h,h'))
\]
By the work of \citet{wolpert1983on_the_symplectic}, we have $\dd{L_\alpha} = \omega_\WP (\cdot, e_\alpha)$, which proves relation \eqref{eq:diff_dual_k_volume_1}, and therefore the statement.
\end{proof}
\end{proposition}

\noindent Since the complex landslide is holomorphic with respect to the complex structure of $\Teich^\hyp(\Sigma^+ \sqcup \Sigma^-)$, an equivalent way to state Proposition \ref{prop:gradient_dual_volume_k} is the following:
\begin{proposition}
Let $M$ be a quasi-Fuchsian manifold and let $h_k$, $h_k' \in \Teich^\hyp(\Sigma^+ \sqcup \Sigma^-)$ denote the hyperbolic metrics $- k \ \I_k$ and $- k (k + 1)^{-1} \ \III_k$. Then the Weil-\hsk Petersson gradient of $V_k^* \circ \Psi_k^{-1}$ coincides, up to a multiplicative factor, with the infinitesimal grafting with respect to the couple $(h_k, h_k')$. In other words,
\[
\grd_\WP{(V_k^* \circ \Psi_k^{-1})} = \sqrt{ - \frac{k + 1}{k}} \ \dv{s} \left. \sgr_s(h_k,h_k') \right|_{s = 0} .
\]
\end{proposition}

The behavior of the third fundamental forms $\III_k$ of the $k$-\hsk surfaces, as $k$ approaches $-1$, is well understood and described by the following Theorem:

\begin{theorem} \label{thm:convergence_k*_surfaces}
Let $(E_n)_n$ be a sequence of hyperbolic ends converging to an hyperbolic end $E$ homeomorphic to $\Sigma \times \R_{\geq 0}$, and let $(k_n)_n$ be any decreasing sequence of numbers converging to $-1$. Then the length spectra $\length_{\III_n}$ converges to the transverse measure $\mu$, where $\III_n$ denotes the third fundamental form of the $k_n$-\hsk surface of $E_n$, and $\mu$ is the bending measured lamination of the concave boundary of $E$.
\end{theorem}

\begin{remark} \label{rmk:duality_hyp_desitter}
Theorem \ref{thm:convergence_k*_surfaces} is in fact a restatement of \cite[Proposition 6.7]{bonsante2013a_cyclic} (see also \cite{belraouti2017asymptotic}). In \cite[Proposition 6.7]{bonsante2013a_cyclic} the authors work with \emph{maximal global hyperbolic spatially compact}  (MGHC) \emph{de Sitter spacetimes}, which relate to the world of hyperbolic ends through the duality between the de Sitter and the hyperbolic space-\hsk forms, as observed by \citet{mess2007lorentz}. In particular, this phenomenon allowed \citet{barbot2011prescribing} to give an alternative proof of the existence of the foliation by $k$-\hsk surfaces.
\end{remark}

We finally have all the elements to prove Proposition \ref{prop:uniform_convergence_diff}.

\begin{proof}[Proof of Proposition \ref{prop:uniform_convergence_diff}]
Let $(M_n)_n$ be a sequence of quasi-Fuchsian manifolds converging to $M$, and let $(k_n)_n$ be a decreasing sequence converging to $-1$. We denote my $h_n$ and $h_n'$ the hyperbolic metrics
\[
h_n \defin - k_n \I_{k_n}, \qquad h_n' \defin - \frac{k_n}{1 + k_n} \III_{k_n},
\]
where $\I_{k_n}$ and $\III_{k_n}$ are the first and third fundamental forms of the $k_n$-\hsk surface $\Sigma^+_{k_n} \sqcup \Sigma^-_{k_n}$ sitting inside $M_n$. By a compactness argument similar to the one described in the proof of Lemma \ref{lem:compact}, the hyperbolic metrics $h_n$ converge to the metric $m$ on the boundary of the convex core of $M$. If we take
\[
\theta_n \defin \sqrt{ - \frac{1 + k_n}{k_n}} ,
\]
then, by Theorem \ref{thm:convergence_k*_surfaces}, the length spectrum of $\theta_n \, \length_{h_n'}$ converges to the bending measure $\mu$ of the boundary of the convex core of $M$. Therefore, applying Theorem \ref{thm:limit_landslides} we obtain that $\theta_n \, l^1(h_n,h_n')|_{h_n}$ converges to $1/2 \ e_\mu |_m$. Combining this with Proposition \ref{prop:gradient_dual_volume_k}, we prove that
\begin{align*}
	\lim_{n \to \infty} \sigma_k(M_n) & = \lim_{n \to \infty} \theta_n \ \omega_\WP(l^1(h_n,h_n'),\cdot) \tag{Proposition \ref{prop:gradient_dual_volume_k}} \\
	& = \frac{1}{2} \omega_\WP(e_\mu, \cdot) \tag{Theorem \ref{thm:limit_landslides}} \\
	& = - \frac{1}{2} \omega_\WP(\cdot, e_\mu) \\
	& = - \frac{1}{2} \dd{(L_\mu)_m}(\cdot)
\end{align*}
where in the last step we applied Wolpert's formula \cite{wolpert1983on_the_symplectic}. By definition (see relation \eqref{eq:def_sigma}), the term $- \frac{1}{2} \dd{(L_\mu)_m}$ coincides with $\sigma(M)$, which concludes the proof of the statement.
\end{proof}


\emergencystretch=1em

\printbibliography[heading=bibintoc,title={References}]

\end{document}